\documentclass{amsart}
\usepackage{amsmath,amsthm,indentfirst}
\usepackage{amssymb}
\usepackage{amsfonts,dsfont}
\usepackage{a4}
\usepackage{amscd}
\usepackage[latin1]{inputenc}
\usepackage{url}
\usepackage{mathtools}
\theoremstyle{plain}
\newtheorem{thm}{Theorem}
\newtheorem{cor}[thm]{Corollary} 
\newtheorem{lem}[thm]{Lemma}
\newtheorem{prop}[thm]{Proposition}

\newenvironment{thm1'}
{\addtocounter{thm}{-1}%
	\begin{thm}}
	{\end{thm}}

\newenvironment{thm'}
{\addtocounter{thm}{-1}%
	\begin{thm}}
	{\end{thm}}

\newenvironment{cor'}
{\addtocounter{thm}{-1}%
	\begin{cor}}
	{\end{cor}}

\newenvironment{lem'}
{\addtocounter{thm}{-1}%
	\begin{lem}}
	{\end{lem}}

\newenvironment{prop'}
{\addtocounter{thm}{-1}%
	\begin{prop}}
	{\end{prop}}

\newtheorem{conjecture}[thm]{Conjecture}

\newtheorem{definition}[thm]{Definition}

\newtheorem{remark}[thm]{Remark}

\def\bbz{\mathbb{Z}}
\def\bbq{\mathbb{Q}}

\def\bba{\mathbb{A}}

\def\bbh{\mathbb{H}}
\def\bbn{\mathbb{N}}
\def\bbg{\mathbb{G}}

\def\bbv{\mathbb{V}}
\def\bbu{\mathbb{U}}

\def\bbl{\mathbb{L}}

\def\bbw{\mathbb{W}}

\def\gcal{\mathcal{G}}

\def\ucal{\mathcal{U}}
\def\vcal{\mathcal{V}}

\def\lcal{\mathcal{L}}
\def\pcal{\mathcal{P}}
\def\wcal{\mathcal{W}}
\def\hfr{\mathfrak{h}}
\def\afr{\mathfrak{a}}

\def\pfr{\mathfrak{p}}
\def\Afr{\mathfrak{A}}

\def\ufr{\mathfrak{u}}
\def\gl{\mathfrak{gl}}
\def\gfr{\mathfrak{g}}

\def\f{\mathfrak{f}}

\def\xbf{\mathbf{x}}

\def\vbf{\mathbf{v}}

\def\ibf{\mathbf{i}}

\DeclareMathOperator\Spec{Spec}

\DeclareMathOperator\SL{SL}
\DeclareMathOperator\GL{GL}

\DeclareMathOperator\Ad{Ad}
\DeclareMathOperator\ad{ad}
\DeclareMathOperator\Lie{Lie}

\DeclareMathOperator{\im}{Im}
\DeclareMathOperator{\pr}{pr}

\newcommand{\wt}[1]{\widetilde{#1}}
\newcommand{\wh}[1]{\widehat{#1}}

\def\h{\hspace{1mm}}

\def\be{\begin{equation}}
\def\ee{\end{equation}}
\def\one{\mathds{1}}

\def\acts{\curvearrowright}

\usepackage{keyval}
\makeatletter
\define@key{setpar}{left}[0pt]{\leftmargin=#1}
\define@key{setpar}{right}[0pt]{\rightmargin=#1}
\define@key{setpar}{both}{\leftmargin=#1\relax\rightmargin=#1}
\makeatother

\newenvironment{narrow}[1][]
{\list{}{\setkeys{setpar}{left,right}%
		\setkeys{setpar}{#1}%
		\listparindent=\parindent
		\topsep=0pt
		\partopsep=0pt
		\parsep=\parskip}\item\relax\hspace*{\listparindent}\ignorespaces}
{\endlist}

\begin{document}
\title{Inducing super-approximation.}

\author{Alireza Salehi Golsefidy}
\address{Mathematics Dept, University of California, San Diego, CA 92093-0112.}
\email{golsefidy@ucsd.edu}

\author{Xin Zhang}
\address{Mathematics Dept, University of Illinois, Urbana, IL 61801.}
\email{xz87@illinois.edu}
\thanks{A. S-G. was partially supported by the NSF grants DMS-1303121 and DMS-1602137, and A. P. Sloan Research Fellowship.}
\date{12/26/2016}
\begin{abstract}
Let $\Gamma_2\subseteq \Gamma_1$ be finitely generated subgroups of $\GL_{n_0}(\bbz[1/q_0])$. For $i=1$ or $2$, let $\bbg_i$ be the Zariski-closure of $\Gamma_i$ in $(\GL_{n_0})_{\bbq}$, $\bbg_i^{\circ}$ be the Zariski-connected component of $\bbg_i$, and let $G_i$ be the closure of $\Gamma_i$ in $\prod_{p\nmid q_0}\GL_{n_0}(\bbz_p)$.  

In this article we prove that, if $\bbg_1^{\circ}$ is the smallest closed normal subgroup of $\bbg_1^{\circ}$ which contains $\bbg_2^{\circ}$ and $\Gamma_2\acts G_2$ has spectral gap, then $\Gamma_1\acts G_1$ has spectral gap.
\end{abstract}
\maketitle
\section{Introduction and the statement of the main results.}
Let $\Gamma$ be a subgroup of a compact, Hausdorff, second countable group $G$. Let $\overline{\Gamma}$ be the closure of $\Gamma$ in $G$. Suppose $\Omega$ is a finite symmetric generating set of $\Gamma$. Let $\pcal_{\Omega}$ be the probability counting measure on $\Omega$, and let
\[
T_{\Omega}:L^2(\overline{\Gamma})\rightarrow L^2(\overline{\Gamma}), \h T_{\Omega}(f):=\pcal_{\Omega}\ast f:=\frac{1}{|\Omega|}\sum_{\omega\in \Omega} L_{\omega}(f),
\]
where $L_{\omega}(f)(g):=f(\omega^{-1}g)$.
Then it is well-known that $T_{\Omega}$ is a self-adjoint operator, $T_{\Omega}(\one_{\overline{\Gamma}})=\one_{\overline{\Gamma}}$ where $\one_{\overline{\Gamma}}$ is the constant function on $\overline{\Gamma}$, and the operator norm $\|T_{\Omega}\|$ is 1. So the spectrum of $T_{\Omega}$ is a subset of $[-1,1]$ and $T_{\Omega}$ sends the space $L^2(\overline{\Gamma})^{\circ}$ orthogonal to the constant functions to itself. Let $T_{\Omega}^{\circ}$ be the restriction of $T_{\Omega}$ to $L^2(\overline{\Gamma})^{\circ}$. Let
\[
\lambda(\pcal_{\Omega};G):=\sup\{|c||\h c \text{ is in the spectrum of the restriction of } T_{\Omega}^{\circ}\}.
\]
 We say the left action $\Gamma\acts G$ of $\Gamma$ on $G$ has {\em spectral gap} if $\lambda(\pcal_{\Omega};G)<1$. 

 It is worth mentioning that, if $\Omega_1$ and $\Omega_2$ are two generating sets of $\Gamma$ and $\lambda(\pcal_{\Omega_1};G)<1$, then $\lambda(\pcal_{\Omega_2};G)<1$. So having spectral gap is a property of the action $\Gamma\acts G$, and it is independent of the choice of a generating set for $\Gamma$.

 The following is the main theorem of this article.
\begin{thm}\label{t:main}
	 Let $\Gamma_2\subseteq \Gamma_1$ be two finitely generated subgroups of $\GL_{n_0}(\bbz[1/q_0])$. For $i=1,2$, let $\bbg_i$ be the Zariski-closure of $\Gamma_i$ in $(\GL_{n_0})_{\bbq}$ for $i=1,2$, and let $\bbg_i^{\circ}$ be the Zariski-connected subgroup of $\bbg_i$. Suppose the smallest closed normal subgroup of $\bbg_1^{\circ}$ which contains $\bbg_2^{\circ}$ is $\bbg_1^{\circ}$. Then, if $\Gamma_2\acts \prod_{p\nmid q_0}\GL_{n_0}(\bbz_p)$ has spectral gap, then $\Gamma_1\acts \prod_{p\nmid q_0}\GL_{n_0}(\bbz_p)$ has spectral gap.
\end{thm}

\begin{cor}\label{c:SimpleCase}
		 Let $\Gamma_2\subseteq \Gamma_1$ be two finitely generated subgroups of $\GL_{n_0}(\bbz[1/q_0])$. Let $\bbg_1$ be the Zariski-closure of $\Gamma_1$ in $(\GL_{n_0})_{\bbq}$, and let $\bbg_1^{\circ}$ be the Zariski-connected subgroup of $\bbg_1$. Suppose $\bbg_1^{\circ}$ is an almost $\bbq$-simple $\bbq$-group, and $\Gamma_2$ is an infinite group.
		 Then, if $\Gamma_2\acts \prod_{p\nmid q_0}\GL_{n_0}(\bbz_p)$ has spectral gap, then $\Gamma_1\acts \prod_{p\nmid q_0}\GL_{n_0}(\bbz_p)$ has spectral gap.
\end{cor}
\begin{proof}
	Since $\Gamma_2$ is infinite, the Zariski-connected component $\bbg_2^{\circ}$ of the Zariski-closure $\bbg_2$ of $\Gamma_2$ in $(\GL_{n_0})_{\bbq}$ is a non-trivial (Zariski-connected) $\bbq$-subgroup of $\bbg_1^{\circ}$. Since $\bbg_1^{\circ}$ is an almost $\bbq$-simple group and $\bbg_2^{\circ}$ is a non-trivial Zariski-connected $\bbq$-subgroup of $\bbg_1^{\circ}$, the smallest normal subgroup of $\bbg_1^{\circ}$ which contains $\bbg_2^{\circ}$ is $\bbg_1^{\circ}$. And so by Theorem~\ref{t:main} claim follows. 
\end{proof}
	Notice that the smallest closed normal subgroup of $\bbg_1^{\circ}$ which contains $\bbg_2^{\circ}$ is $\bbg_1^{\circ}$ if and only if the restriction of any non-trivial representation $\rho:\bbg_1^{\circ}\rightarrow (\GL_{m})_{\bbq}$ to $\bbg_2^{\circ}$ is still non-trivial. 

\begin{remark}
	The first result of this kind goes back to the work of Burger and Sarnak~\cite{BS}. Their result implies Theorem~\ref{t:main} when $\Gamma_i$'s are integral points of (fixed embeddings of) two Zariski-connected semisimple $\bbq$-groups $\bbg_i$'s. 
\end{remark}

Theorem~\ref{t:main} has an immediate application in explicit construction of expanders. Let us quickly recall that a family of $d$-regular graphs $X_i$ is called a family of expanders if the size of vertices $|V(X_i)|$ goes to infinity and there is a positive number $\delta_0$ such that for any subset $A$ of $V(X_i)$ we have
\[
\frac{|e(A,V(X_i)\setminus A)|}{\min(|A|,|V(X_i)\setminus A|)}>\delta_0,
\]
where $e(A,B)$ is the set of edges that connect a vertex in $A$ to a vertex in $B$. Expanders have a lot of applications in theoretical computer science (see \cite{HLW} for a survey on such applications).

It is well-known that Theorem~\ref{t:main} is equivalent to the following theorem (see \cite[Remark 15]{SG:SAI} or \cite[Section 4.3]{Lub})

\begin{thm1'}\label{t:main'}
	Let $\Omega_1$ and $\Omega_2$ be two finite symmetric subsets of $\GL_{n_0}(\bbz[1/q_0])$. For $i=1,2$, let $\Gamma_i$ be the group generated by $\Omega_i$, and $\bbg_i^{\circ}$ be the Zariski-connected component of the Zariski-closure of $\Gamma_i$ in $(\GL_{n_0})_{\bbq}$. Suppose $\Gamma_2\subseteq \Gamma_1$ and the normal closure of $\bbg_2^{\circ}$ in $\bbg_1^{\circ}$ is $\bbg_1^{\circ}$. Then if the family of Cayley graphs $\{{\rm Cay}(\pi_q(\Gamma_2),\pi_q(\Omega_2)\}_{\gcd(q,q_0)=1}$ is a family of expanders, then  $\{{\rm Cay}(\pi_q(\Gamma_1),\pi_q(\Omega_1)\}_{\gcd(q,q_0)=1}$ is a family of expanders.
\end{thm1'}

\begin{definition}
	Let $\Omega$ be a finite symmetric subset of $\GL_{n_0}(\bbz[1/q_0])$. We say that the group $\Gamma=\langle \Omega\rangle$ generated by $\Omega$ has super-approximation with respect to a subset $C$ of positive integers if the family of Cayley graphs $\{{\rm Cay}(\pi_m(\Gamma),\pi_m(\Omega)\}_{m\in C}$ is a family of expanders.
	
	We simply say $\Gamma\subseteq \GL_{n_0}(\bbz[1/q_0])$ has super-approximation if it has super-approximation with respect to $\{q\in \bbz^+|\gcd(q,q_0)=1\}$.
\end{definition}

In the past decade there has been a surge in proving that super-approximation is a Zariski-topological property (see~\cite{BG1}-\cite{BV} and \cite{SG:SAI}-\cite{SGV}). By now we know that
\begin{enumerate}
	\item A finite generated Zariski-dense subgroup of $\SL_{n_0}(\bbz)$ has super-approximation with respect to $\bbz^+$ (see~\cite[Theorem 1]{BV}).
	\item A finitely generated subgroup $\Gamma$ of $\GL_{n_0}(\bbz[1/q_0])$ has super-approximation with respect to
	\[
	\{q^{m_0}|\h q \text{ is a square-free integer }, \gcd(q,q_0)=1\}
	\]
	 if and only if $\bbg^{\circ}=[\bbg^{\circ},\bbg^{\circ}]$ where $\bbg^{\circ}$ is the Zariski-connected component of the Zariski-closure of $\Gamma$. (see~\cite[Theorem 1]{SGV} and \cite[Theorem 1]{SG:SAII}).
	\item  A finitely generated subgroup $\Gamma$ of $\GL_{n_0}(\bbz[1/q_0])$ has super-approximation with respect to
	\[
	\{p^n|\h n\in \bbz^+, p \text{ is a prime which does not divide } q_0\}
	\]
	 if and only if $\bbg^{\circ}=[\bbg^{\circ},\bbg^{\circ}]$ where $\bbg^{\circ}$ is the Zariski-connected component of the Zariski-closure of $\Gamma$. (see~\cite[Theorem 1]{SG:SAII}).
\end{enumerate}
The following is the main conjecture on this subject.
\begin{conjecture}
	 A finitely generated subgroup $\Gamma$ of $\GL_{n_0}(\bbz[1/q_0])$ has super-approximation
	if and only if $\bbg^{\circ}=[\bbg^{\circ},\bbg^{\circ}]$ where $\bbg^{\circ}$ is the Zariski-connected component of the Zariski-closure of $\Gamma$.
\end{conjecture}

Relaxing the condition on the set of possible residues is crucial in some of the results where super-approximation is used in combination with {\em large sieve} or {\em thermodynamical} techniques (for instance see~\cite{BK, BKM, MOW}).

Theorem~\ref{t:main} is about inducing super-approximation property from a subgroup with {\em large} Zariski-closure to the group itself. So since we know (infinite) arithmetic groups in semisimple $\bbq$-groups, i.e. $\bbz[1/q_0]$-points in a semisimple $\bbq$-group, have super-approximation, we get the following corollary.

\begin{cor}\label{c:arithmetic}
	Let $\Gamma$ be a finitely generated subgroup of $\GL_{n_0}(\bbz[1/q_0])$. Suppose the Zariski-closure of $\Gamma$ in $(\GL_{n_0})_{\bbq}$ is an almost $\bbq$-simple group. Suppose further that $\Gamma$ contains an infinite arithmetic subgroup of some semisimple group. Then $\Gamma$ has super-approximation.
\end{cor}

Using the mentioned result of Bourgain and Varj\'{u}~\cite[Theorem 1]{BV}, we get the following corollary of Theorem~\ref{t:main}.

\begin{cor}\label{c:BV}
	Let $\Gamma$ be a finitely generated subgroup of $\GL_{n_0}(\bbz[1/q_0])$. Suppose the Zariski-closure of $\Gamma$ in $(\GL_{n_0})_{\bbq}$ is an almost $\bbq$-simple group. Suppose further that there is a subgroup $\Lambda\subseteq \Gamma\cap \GL_{n_0}(\bbz)$ whose Zariski-closure is isomorphic to $(\SL_m)_{\bbq}$ for some positive integer $m$. Then $\Gamma$ has super-approximation. 
\end{cor}
\begin{proof}
	Let $\bbh$ be the Zariski-closure of $\Lambda$ in $(\GL_{n_0})_{\bbq}$. By the assumption, there is a $\bbq$-isomorphism $\rho:\bbh\rightarrow (\SL_m)_{\bbq}$. By \cite[Lemma 13]{SG:SAII}, there is $g\in \SL_m(\bbq)$ such that $\rho(\Lambda)\subseteq g\SL_m(\bbz)g^{-1}$. And so $g^{-1}\rho(\Lambda)g$ is a finitely generated Zariski-dense subgroup of $\SL_m(\bbz)$. Hence by \cite[Theorem 1]{BV} and \cite[Remark 15]{SG:SAI} we have that $g^{-1}\rho(\Lambda)g \acts \prod_p \SL_m(\bbz_p)$ has spectral gap. 
	
	Now let $\overline{\Lambda}$ be the closure of $\Lambda$ in $\prod_p \GL_{n_0}(\bbz_p)$, and let $\bba$ be the ring of adeles of $\bbq$. Since $\rho$ induces a topological isomorphism $\rho:\bbh(\bba)\rightarrow \SL_m(\bba)$, $g^{-1}\rho(\overline{\Lambda})g$ is the closure of $g^{-1}\rho(\Lambda)g$ in $\prod_p \SL_m(\bbz_p)$. Since $g^{-1}\rho(\Lambda)g \acts \prod_p \SL_m(\bbz_p)$ has spectral gap, we get that $\rho(\Lambda)\acts \rho(\overline{\Lambda})$ has spectral gap. And so $\Lambda\acts \overline{\Lambda}$ has spectral gap. Therefore by Theorem~\ref{t:main} and \cite[Remark 15]{SG:SAI} the claim follows.
\end{proof}

This kind of result was first obtained by Varj\'{u} in the appendix of~\cite{BK} where he proved a special case of Corollary~\ref{c:arithmetic} for the group of symmetries of an Apollonian packing. More recently in \cite[Theorem 1.3]{FSZ} a special case of Corollary~\ref{c:BV} is proved where the Zariski-closure of $\Gamma$ is assumed to be isomorphic to the restriction of scalars $R_{k/\bbq}((\SL_2)_{\bbq})$ of $(\SL_2)_{\bbq}$ for a finite extension $k/\bbq$. 
 
\subsection{Outline of the proof.} Here is an outline of the main ideas of the proof. 

{\bf Step 0. (Initial preparation)} Preliminary reductions: it is showed that one can essentially work under the extra assumptions that the Zariski-closure $\bbg_i$ of $\Gamma_i$ in $(\GL_{n_0})_{\bbq}$ is Zariski-connected for $i=1,2$; and $\bbg_1$ is simply-connected (see Section~\ref{s:FinalProof}).
  
{\bf Step 1. (Reduction to an Adelic Bounded Generation)}  Using Varj\'{u}'s Lemma~\cite[Lemma A.2]{BK} and \cite[Lemma 16]{SG:SAI}, the proof of Theorem~\ref{t:main} is reduced to an adelic bounded generation statement (see Theorem~\ref{t:BGAdelic}). The following is a variant of \cite[Lemma A.2]{BK}.
\begin{lem}[Varj\'{u}'s Lemma]
	Suppose $G$ is a finite group, $H$ is a subgroup of $G$, and 
	\[
		G=g_1Hg_1^{-1}\cdot g_2Hg_2^{-1}\cdots g_nHg_n^{-1}
	\]
   for some $g_i\in G$. Let $\Omega$ be a symmetric generating set of $H$. Then
	\[
	\lambda(\pcal_{\Omega'};G)\le f(|\Omega|,\lambda(\pcal_{\Omega};H),n)<1,
	\]
	where $\Omega':=\bigcup_{i=1}^n g_i\Omega g_i^{-1}$ and  $f:\bbz^+\times [0,1] \times \bbz^+ \rightarrow [0,1)$.
\end{lem} 
By \cite[Lemma 16]{SG:SAI}, it is enough to get a spectral gap for a subgroup of finite index. Hence altogether it is enough to prove the following (see Theorem~\ref{t:BGAdelic}):

{\bf Adelic Bounded Generation:} {\em there are $\gamma_1,\ldots,\gamma_m\in \Gamma_1$ such that $\gamma_1\wh{\Gamma}_2\gamma_1^{-1}\cdots \gamma_m\wh{\Gamma}_2\gamma_m^{-1}$ is an open subgroup of $\wh{\Gamma}_1$ where $\wh{\Gamma}_i$ is the closure of $\Gamma_i$ in $\prod_{p\nmid q_0}\GL_{n_0}(\bbz_p)$.}

To get the above mentioned Adelic Bounded Generation result, we prove many bounded generation results from various angles: Lie algebras; Zariski topology; and $p$-adic topology (with certain uniformity on $p$). Here is a bit more detailed description of these steps.

{\bf Step 2. (Generating the Lie algebra)}  We prove that $\Lie(\bbg_2)(\bbq)$ generates $\Lie(\bbg_1)(\bbq)$ as a $\Gamma_1$-module. 

{\bf Step 3. (Bounded Generation: Zariski topology)} The infinitesimal result proved in the previous step shows that $(g_1,\ldots,g_n)\mapsto \gamma_1g_1\gamma_1^{-1}\cdots \gamma_ng_n\gamma_ng^{-1}$ is a geometrically dominant morphism from $\bbg_2\times \cdots \times \bbg_2$ to $\bbg_1$ for suitable $\gamma_i$'s in $\Gamma_1$. Looking at the scheme theoretic closure of $\Gamma_i$, we deduce that for almost all the geometric fibers $\bbg_i^{(p)}$ we still get dominant morphisms (see Proposition~\ref{p:BGZariskiTopology}).

{\bf Step 4. (Bounded Generation: $p$-adic topology)} Based on a quantitative open function theorem for $p$-adic analytic functions proved in~\cite[Lemma 45']{SG:SAII}, we show the following $p$-adic topological bounded generation with certain uniformity on $p$ (see Proposition~\ref{p:LargeCongruence}):

{\em there are a positive integer $N$ and $\gamma_1,\ldots,\gamma_n\in\Gamma_1$ such that 
$\Gamma_{1,p}[p^N]\subseteq \gamma_1\Gamma_{2,p}\gamma_1^{-1}\cdots \gamma_n\Gamma_{2,p}\gamma_n^{-1}$ where $\Gamma_{i,p}$ is the closure of $\Gamma_i$ in $\GL_{n_0}(\bbz_p)$ and $\Gamma_{1,p}[p^N]:=\Gamma_{1,p}\cap 1+p^N\gl_{n_0}(\bbz_p)$.
}
  
{\bf Step 5. (Bounded Generation: modulo $p$, for large primes)} Step 3 gave us certain dominant morphisms. A result of Pink and R\"{u}tsche~\cite[Proposition 2.5]{PR} helps us to deduce that the image of these morphisms applied to the $\f_p$-points of the underlying varieties is {\em large} \footnote{In this note $\f_q$ denotes the finite field with $q$ elements.}. And so by a result of Gowers~\cite{Gow} (also see~\cite[Corollary 1]{NP}) we get that the three fold multiple of the image of such a morphism is the entire $\f_p$-points of the considered group. So altogether using Nori's strong approximation~\cite[Theorem 5.4]{Nor} we get the following (see Lemma~\ref{l:BGModP})

{\em there are $\gamma_1,\ldots,\gamma_n\in \Gamma_1$ such that for large enough $p$ we have $\pi_p(\gamma_1\Gamma_{2,p}\gamma_1^{-1}\ldots\gamma_n\Gamma_{2,p}\gamma_n^{-1})=\pi_p(\Gamma_{1,p})$, where $\pi_p$ is the group homomorphism induced by the quotient map $\bbz_p\rightarrow \f_p$.}

{\bf Step 6. (Proving the Adelic Bounded Generation)} Using the truncated logarithmic maps (see Lemma~\ref{l:ThekthGrade}), Step 2, taking multiple commutators, and Step 5, we can generate the first $N$ {\em $p$-adic layers} of $\Gamma_{1,p}$ for large enough $p$. This result together with Step 4 give us the Adelic Bounded Generation.
 
\section*{Acknowledgment}
The first author would like to thank Amir Mohammadi for many valuable conversations. The second author is grateful to Alex Kontorovich for useful discussion on this topic. Authors would like to thank Alex Kontorovich,  Alex Lubotzky, Amir Mohammadi, Peter Sarnak, and Peter Varj\'{u} for their comments on the earlier version of this note. 

\section{Preliminary results.}
In this section, we gather some of the needed well-known results and adapt them to our setting. 

\subsection{Recalling basic analytic properties of $\bbq_p$-points of an algebraic group.}\label{s:AlgebraicAnalytic}

For the convenience of the reader, in this section some of the well-known analytic properties of the $\bbq_p$-points $\bbg(\bbq_p)$ of a linear algebraic $\bbq_p$-group $\bbg\subseteq (\GL_{n_0})_{\bbq_p}$ is recalled. 
For instance, we will present a detailed argument of why the logarithmic function induces bijection between $G_c:=\bbg(\bbq_p)\cap (I+p^c\gl_{n_0}(\bbz_p))$ and $\Lie(\bbg)(\bbq_p)\cap p^c\gl_{n_0}(\bbz_p)$ if $c$ is a positive integer and $c>1$ when $p=2$. 

We start by summarizing the basic properties of the logarithmic and the exponential functions. In what follows $c_0$ is 1 if $p$ is odd and it is 2 if $p=2$. For any prime $p$, the exponential function 
\[
\exp:p^{c_0}\gl_{n_0}(\bbz_p)\rightarrow \GL_{n_0}(\bbz_p)[p^{c_0}],
\h
\exp(x):=\sum_{i=0}^{\infty}x^i/i!
\]
and the logarithmic function 
\[
\log:\GL_{n_0}(\bbz_p)[p^{c_0}]\rightarrow p^{c_0} \gl_{n_0}(\bbz_p),
\h
\log g:=-\sum_{i=1}^{\infty}(I-g)^i/i
\]
are well-defined analytic functions, and inverse of each other, where 
\[
\GL_{n_0}(\bbz_p)[p^{c_0}]:=\ker(\GL_{n_0}(\bbz_p)\xrightarrow{\pi_{p^{c_0}}}\GL_{n_0}(\bbz/p^{c_0}\bbz));
\]
 moreover we have 
\[
\|\exp(x)-I\|_p=\|x\|_p,\h \text{and}\h \|\log(g)\|_p=\|g-I\|_p,
\]
for any $x\in p^{c_0}\gl_{n_0}(\bbz_p)$ and $g\in \GL_{n_0}(\bbz_p)[p^{c_0}]$. Therefore for any integer $c\ge c_0$ the exponential and the logarithmic functions induce bijections between $\GL_{n_0}(\bbz_p)[p^c]$ and $p^c\gl_{n_0}(\bbz_p)$. 

The next lemma is an easy corollary of the chain rule, but a direct detailed argument is presented. 

\begin{lem}\label{l:AnalyticAlgebraic}
	Let $p$ be a prime, and  $x\in p^{c_0}\gl_{n_0}(\bbz_p)$ where $c_0=2$ when $p=2$ is odd and $c_0=1$ otherwise. Suppose $f\in \bbq[X_{ij}]$ is a polynomial on the entries of $n_0$-by-$n_0$ matrices. Then 
	\[
	\lim_{n\rightarrow \infty}\frac{f(\exp(p^n x))-f(I)}{p^n}=(df)_I(x),
	\]	 
	where 	
	\[
	(df)_I(Y_{ij}):=\sum_{ij}\frac{\partial f}{\partial X_{ij}}(I)Y_{ij}.
	\]
\end{lem}
\begin{proof}
	After multiplying $f$ by a suitable non-zero integer, we can and will assume that $f$ has integer coefficients. To make the symbols a bit more clear, we view $f$ as a polynomial of $n_0^2$ variables $T_i$. Here $J$ denotes a multi-index, i.e. $J=(j_1,\ldots,j_{n_0^2})$ where $j_i$ are non-negative integers. We denote the symbolic higher-order partial derivatives of a polynomial $f$ by $\partial_J f$. For a multi-index $J$, let ${\bf T}^{J}$ be the monomial $\prod_i T_i^{j_i}$. By Taylor expansion we have
	\be\label{e:Taylor}
	f(I+[T_i])=\sum_J \partial_J f(I) {\bf T}^J=f(I)+(df)_I(T_i)+\sum_{J,\|J\|_1>1}\partial_J f(I) {\bf T}^J,
	\ee
	where $[T_i]$ is the $n_0$-by-$n_0$ matrix whose $ij$-entry is $T_{n_0(i-1)+j}$ and $\|J\|_1:=\sum_i |j_i|$. For $x\in p^{c_0}\gl_{n_0}(\bbz_p)$ and a positive integer $n$, let $t_i:=T_i(\exp(p^n x)-I)$ be the $i$-th component of $\exp(p^n x)-I$ in the above ordering. So 
	\be\label{e:1}
	\left\|\frac{f(\exp(p^n x))-f(I)-(df)_I(t_i)}{p^n}\right\|_p=
	\left\|\frac{\sum_{J,\|J\|_1>1} \partial_Jf(I) {\bf t}^J}{p^n}\right\|_p\le \frac{\|\exp(p^n x)-I\|_p^{2}}{\|p^n\|_p}=\frac{\|p^nx\|_p^{2}}{\|p^n\|_p}\le \|p^n\|_p.
	\ee  
	On the other hand, $t_i=\sum_{k=1}^{\infty} (p^{kn}/k!)T_i(x^k)$ where $T_i(\cdot)$ is the function which gives the $i$-th component of a matrix in the above ordering. Thus we get
	\be\label{e:2}
	\left\|\frac{(df)_I(t_i)}{p^n}-(df)_I(x)\right\|_p=\|\sum_{k=2}^{\infty}(p^{(k-1)n}/k!)(df)_I(x^k)\|_p\le \|p^n\|_p.
	\ee
	Hence by (\ref{e:1}) and (\ref{e:2}) we get 
	\[
	\left\|\frac{f(\exp(p^n x))-f(I)}{p^n}-(df)_I(x)\right\|_p\le \|p^n\|_p,
	\]
	which implies our claim.
\end{proof}
\begin{cor}\label{l:LogInLie}
	Let $\bbg\subseteq (\GL_{n_0})_{\bbq_p}$ be a given embedding of a Zariski-connected $\bbq_p$-group $\bbg$. Then for any positive integer $c$ which is more than 1 for $p=2$, the logarithmic function
	\[
	\log: \bbg(\bbq_p)\cap \GL_{n_0}(\bbz_p)[p^c]\rightarrow \Lie(\bbg)(\bbq_p)\cap p^c\gl_{n_0}(\bbz_p)
	\]
	is a well-defined injection.  
\end{cor}
\begin{proof}
	We already know that $\log:\GL_{n_0}(\bbz_p)[p^c]\rightarrow p^c\gl_{n_0}(\bbz_p)$ is a bijection. So it is enough to show that for any $g\in G_c:=\bbg(\bbq_p)\cap \GL_{n_0}(\bbz_p)[p^c]$ we have $\log g\in \Lie(\bbg)(\bbq_p)$. 
	
	Next we notice that the natural embedding $g\mapsto {\rm diag}(g,(\det g)^{-1})$ of $(\GL_{n_0})_{\bbz_p}$ into $(\SL_{n_0+1})_{\bbz_p}$ sends $\GL_{n_0}(\bbz_p)[p^c]$ to $\SL_{n_0+1}(\bbz_p)[p^c]$ and commutes with the logarithmic function. So we can and will consider $\bbg$ as a subgroup of $(\SL_{n_0+1})_{\bbq_p}$. So $\bbg$ as a closed subset of $(\SL_{n_0+1})_{\bbq_p}$ is given by relations $f_i$, where $f_i$ are polynomials on the entries of $(n_0+1)$-by-$(n_0+1)$ matrices. Hence
	\[
	\Lie(\bbg)(\bbq_p)=\{x\in \gl_{n_0+1}(\bbq_p)|\h (df_i)_I(x)=0\}.
	\]
	By Lemma~\ref{l:LogInLie}, we have
	\[
	(df_i)_I(\log g)=\lim_{n\rightarrow \infty}\frac{f_i(\exp(p^n \log g))-f_i(I)}{p^n}=\lim_{n\rightarrow \infty}\frac{f_i(g^{p^n})-f_i(I)}{p^n}=0,
	\]
	which implies $\log g\in \Lie(\bbg)(\bbq_p)$.
\end{proof}
\begin{prop}\label{p:AnalyticIsAlgebraic}
	Let $\bbg\subseteq (\GL_{n_0})_{\bbq_p}$ be a given embedding of a Zariski-connected $\bbq_p$-group $\bbg$. Then for a positive integer $c$, which is at least $2$ if $p=2$, the logarithmic function
		\[
	\log: \bbg(\bbq_p)\cap \GL_{n_0}(\bbz_p)[p^c]\rightarrow \Lie(\bbg)(\bbq_p)\cap p^c\gl_{n_0}(\bbz_p)
	\]
	is a bijection; and so the restriction of the exponential function is its inverse. 
\end{prop}
\begin{proof}
	First we prove the claim when $c$ is large depending on the embedding of $\bbg$. 
	
	By~\cite[Chapter I, Lemma 2.5.1 (i)]{Mar}, we have that the dimension of $\bbg(\bbq_p)$ as a $\bbq_p$-analytic manifold is the same as $\dim \bbg$. Since, by Lemma~\ref{l:LogInLie}, the restriction of the logarithmic function is an analytic immersion of $G_2:=\bbg(\bbq_p)\cap \GL_{n_0}(\bbz_p)[p^2]$ into $\Lie(\bbg)(\bbq_p)$ and as $p$-adic analytic manifolds $G_2$ has the same dimension as $\Lie(\bbg)(\bbq_p)$, we have that for some positive integer $c_1$
	\[
		p^{c_1}\gl_{n_0}(\bbz_p)\cap \Lie(\bbg)(\bbq_p)\subseteq \log(G_2).
	\]
	This implies that $\exp(p^{c_1}\gl_{n_0}(\bbz_p)\cap \Lie(\bbg)(\bbq_p))\subseteq \bbg(\bbq_p)$. Hence
for any $c\ge c_1$ we have
	\be\label{e:LargeOpenSet}
		\exp(p^{c}\gl_{n_0}(\bbz_p)\cap \Lie(\bbg)(\bbq_p))\subseteq G_c,
	\ee
So by (\ref{e:LargeOpenSet}) and Lemma~\ref{l:LogInLie} we get that $\log:G_c\rightarrow \Lie(\bbg)(\bbq_p)\cap p^c\gl_{n_0}(\bbz_p)$ is a bijection.  

Suppose the closed immersion of $\bbg$ is given by the ideal $I_{\bbg}\triangleleft \bbq[\GL_{n_0}]$. For any $f\in I_{\bbg}$ the composite function $f\circ \exp$ defines a $p$-adic analytic function on $\Lie(\bbg)(\bbq_p)\cap p^{c_0}\gl_{n_0}(\bbz_p)$ where $c_0=1$ if $p>2$ and $c_0=2$ if $p=2$. We have proved that this analytic function is identically zero on the open set $\Lie(\bbg)(\bbq_p)\cap p^{c}\gl_{n_0}(\bbz_p)$ if $c$ is large enough. Hence it is zero, which implies 
\[
\exp(\Lie(\bbg)(\bbq_p)\cap p^{c_0}\gl_{n_0}(\bbz_p))\subseteq \bbg(\bbq_p).
\]
Therefore $\log: \bbg(\bbq_p)\cap \GL_{n_0}(\bbz_p)[p^{c_0}]\rightarrow \Lie(\bbg)(\bbq_p)\cap p^{c_0}\gl_{n_0}(\bbz_p)$ is a bijection.
\end{proof}

\subsection{A remark on certain flat models of an algebraic group.}
In this work we need to work with certain group schemes over either $\bbz[1/q_0]$ or $\bbz_p$. In order to treat them at the same time, only in this section, we let $A$ be a PID and $F$ be its quotient field. 

Let $\bbg$ be a linear algebraic group defined over $F$. For a fixed $F$-embedding $\rho:\bbg\rightarrow (\GL_{n_0})_F$, $\bbg_{\rho}$ denotes the image of $\rho$.  Now we view $(\GL_{n_0})_F$ as the generic fiber of the group scheme $(\GL_{n_0})_A$ and let $\gcal_{\rho}$ be the Zariski-closure of $\bbg_{\rho}$ in $(\GL_{n_0})_A$. 

To clarify the previous paragraph, let $F[\GL_{n_0}]$ be the ring of regular functions of $(\GL_{n_0})_F$ and $I_{\bbg_{\rho}}$ be the defining ideal of $\bbg_{\rho}$ (which means the scheme structure of $\bbg_{\rho}$ is $\Spec(F[\GL_{n_0}]/I_{\bbg_{\rho}})$); then the ring of regular functions $A[\gcal_{\rho}]$ of $\gcal_{\rho}$ is $A[\GL_{n_0}]/(A[\GL_{n_0}]\cap I_{\bbg_{\rho}})$ where $A[\GL_{n_0}]$ is the ring of regular functions of $(\GL_{n_0})_A$.

\begin{lem}\label{l:FlatModel}
	In the above setting, the generic fiber of $\gcal_{\rho}$ is isomorphic to $\bbg_{\rho}$; and, for any $A$-algebra $R$ with free $A$-module structure, $\gcal_{\rho}(R)$ can be naturally identified with $\bbg_{\rho}(R\otimes_A F)\cap \GL_{n_0}(R)$. (Here, since $R$ is a free $A$-modules, we can and will identify $R$ with a subring of $R\otimes_A F$ through $x\mapsto x\otimes 1$.) In particular, $\Lie (\gcal_{\rho})(R)$ can be naturally identified with $\Lie(\bbg_{\rho})(R\otimes_A F)\cap \gl_{n_0}(R)$.
\end{lem}
\begin{proof}
	Let $S:=A\setminus \{0\}$ and $\Afr:=A[\GL_{n_0}]$. Then $F=S^{-1}A$ and $\Afr\otimes_A F\simeq S^{-1}\Afr= F[\GL_{n_0}]$. 
	Let $I:=\Afr\cap  I_{\bbg_{\rho}}$. Then $S^{-1}I= I_{\bbg_{\rho}}$. So 
	$(\Afr/I)\otimes_A F\simeq S^{-1}(\Afr/I)\simeq S^{-1}\Afr/S^{-1}I=F[\GL_{n_0}]/ I_{\bbg_{\rho}}=F[\bbg_{\rho}]$.
	
	For any $\phi\in \gcal_{\rho}(R):={\rm Hom}_{A-{\rm alg.}}(A[\gcal_{\rho}],R)$ and $a\in A\setminus \{0\}$, we have $\phi(a)=a \phi(1)\neq 0$ as $R$ is a free $A$-module. Hence $S^{-1}\phi: S^{-1}A[\gcal_{\rho}]\rightarrow S^{-1}R$ is well-defined. As we discussed above $S^{-1}A[\gcal_{\rho}]=F[\bbg_{\rho}]$; and $S^{-1}R\simeq R\otimes_A F$. On the other hand, $\phi$ can be lifted to an $A$-algebra homomorphism from $\Afr$ to $R$. So we get the desired point in $\bbg_{\rho}(R\otimes_A F)\cap \GL_{n_0}(R)$.   
	
	Now suppose $\wh{\phi}\in\bbg_{\rho}(R\otimes_A F)\cap \GL_{n_0}(R)$. This means $\wt{\phi}:F[\GL_{n_0}]\rightarrow R\otimes_A F$, $\wt{\phi}(x):=\wh{\phi}(x+I_{\bbg_{\rho}})$, sends the generators $t_{ij}, t$ of $F[\GL_{n_0}]$ to $R$. 
	
	As $A$ is a PID, $A[\gcal_{\rho}]$ is a free $A$-module; and so we can and will identify $A[\gcal_{\rho}]$ by a subring of $S^{-1}A[\gcal_{\rho}]=F[\bbg_{\rho}]$. Now let $\phi$ be the $A$-algebra homomorphism which is the restriction of $\wh{\phi}$ to $A[\gcal_{\rho}]$. Since the standard generators of $F[\GL_{n_0}]$ are sent to $R$ by $\wh{\phi}$, we can let $R$ to be the codomain of $\phi$. Hence we get an $A$-algebra homomorphism $\phi:A[\gcal_{\rho}]\rightarrow R$, i.e. $\phi\in \gcal_{\rho}(R)$.
	
	To get the last part, we first notice that $R[T]/\langle T^2\rangle=R\oplus R\overline{T}$ is a free $A$-module. Hence by the previous part and the definition of Lie algebra we get the following commutative diagram
\[
\begin{array}{cccccccc}
1\rightarrow & \Lie(\gcal_{\rho})(R) &	 \rightarrow & \gcal_{\rho}(R[\overline T]) & \rightarrow & \gcal_{\rho}(R) & \rightarrow 1\\
& & &\rotatebox{-90}{$\xrightarrow{\sim}$} & &\rotatebox{-90}{$\xrightarrow{\sim}$} & 
\\
1\rightarrow & \Lie(\bbg_{\rho})(R\otimes_A F)\cap \gl_{n_0}(R) &	 \rightarrow & \bbg_{\rho}(R[\overline T]\otimes_A F) \cap \GL_{n_0}(R[\overline T]) & \rightarrow & \bbg_{\rho}(R\otimes_A F)\cap \GL_{n_0}(R) & \rightarrow 1,
\end{array}
\]
which implies the last part. 
	 
\end{proof}

\section{Generating the Lie algebra as a module under the adjoint action.}
In this section, we study perfect algebraic groups and their adjoint representation. The main result of this section is Proposition~\ref{p:GeneratingLieAlgebra}.

Let's recall that the normal closure of an algebraic subgroup $\bbg_2$ of $\bbg_1$ is the smallest normal closed subgroup of $\bbg_2$ in $\bbg_1$. 

\begin{prop}
	\label{p:GeneratingLieAlgebra}
	Let $\bbg_2\subseteq \bbg_1$ be Zariski-connected $\bbq$-groups. Suppose $\bbg_2$ is a perfect $\bbq$-subgroup of $\bbg_1$, and further assume that the normal closure of $\bbg_2$ in $\bbg_1$ is $\bbg_1$. Let $\gfr_i:=\Lie(\bbg_i)(\bbq)$ be the $\bbq$-structure of the Lie algebras of $\bbg_i$. Then $\gfr_2$ generates $\gfr_1$ as an $\bbg_1(\bbq)$-module under the adjoint representation.    
\end{prop}

\begin{lem}\label{l:Perfect}
		Let $\bbg_2\subseteq \bbg_1$ be Zariski-connected $\bbq$-groups. Suppose $\bbg_2$ is a perfect $\bbq$-subgroup of $\bbg_1$, and further assume that the normal closure of $\bbg_2$ in $\bbg_1$ is $\bbg_1$. Then $\bbg_1$ is perfect.
\end{lem}
\begin{proof}
	Let $f:\bbg_1\rightarrow \bbg_1/[\bbg_1,\bbg_1]$ be the natural quotient map. Since $\bbg_2$ is perfect, $\bbg_2$ is a subgroup of the kernel of $f$. Since $\ker(f)$ is a normal subgroup of $\bbg_1$ and the normal closure of $\bbg_2$ in $\bbg_1$ is $\bbg_1$, we get that $\ker(f)=\bbg_1$. Hence $\bbg_1$ is perfect.
\end{proof}
The following lemma has been proved in \cite[Lemma 20]{SG:SAII} for large finite fields. It is reproved here for the convenience of the reader. Lemma~\ref{l:ReductionToAbelianUnipotentRadical} enables us to reduce the proof of Proposition~\ref{p:GeneratingLieAlgebra} to the case where the unipotent radical of $\bbg_1$ is abelian.

\begin{lem}
	\label{l:ReductionToAbelianUnipotentRadical}
	Let $\bbh$ be a Zariski-connected semisimple $\bbq$-group and $\bbu$ be a unipotent $\bbq$-group. Suppose $\bbh$ acts on $\bbu$, and let $\bbg:=\bbh\ltimes \bbu$. Let $A_{\bbg}$ be the $\bbq$-span of $\Ad(\bbg(\bbq))\subseteq {\rm End}_{\bbq}(\Lie(\bbg)(\bbq))$ and $\afr_{\bbu}$ be the ideal of $A_{\bbg}$ generated by $\{\Ad(u)-1|\h u\in \bbu(\bbq)\}$. Then
	\begin{enumerate}
		\item The Jacobson radical $J(A_{\bbg})$ of $A_{\bbg}$ is equal to $\afr_{\bbu}$.
		\item $[\ufr,\ufr]\subseteq J(A_{\bbg}) \gfr$ where $\ufr:=\Lie(\bbu)(\bbq)$ and $\bbg:=\Lie(\bbg)(\bbq)$.  
	\end{enumerate}
\end{lem}
\begin{proof}
	Since $\bbu$ is unipotent, $\afr_{\bbu}$ is a nilpotent ideal. Hence $\afr_{\bbu}\subseteq J(A_{\bbg})$. Let $A_{\bbh}$ be the $\bbq$-span of $\Ad(\bbh(\bbq))$ as a subset of ${\rm End}_{\bbq}(\gfr)$. Since $\bbh$ is semisimple, $A_{\bbh}$ is a semisimple algebra. Moreover, as a $\bbq$-algebra, we have $A_{\bbg}/\afr_{\bbu}\simeq A_{\bbh}$. Overall we get $J(A_{\bbg})=\afr_{\bbu}$. 
	
		Since $\bbu$ is unipotent, $\log$ and $\exp$ define $\bbq$-morphisms between $\bbu$ and its Lie algebra. For any $n\in \bbz$, $x,y\in \ufr$ we have 
		\be\label{e:ExpAd}
		\Ad(\exp(nx))(y)=\exp(n\ad(x))(y).
		\ee
		Thus for $n\in \bbz^+$, $x,y\in \ufr$ we have
		\be\label{e:ExpAd2}
		n^{-1}(\Ad(\exp(nx))(y)-y)=[x,y]+\sum_{i=1}^{\dim_{\bbq}\ufr} \frac{n^i}{i!} \ad(x)^i(y)\in J(A_{\bbg})\gfr.
		\ee
		Therefore by the Vandermonde determinant we have that $[x,y]\in J(A_{\bbg})\gfr$. 
\end{proof}
\begin{lem}\label{l:Semisimple}
	Let $\bbh$ be a Zariski-connected, semisimple $\bbq$-group, and $\hfr=\Lie(\bbh)(\bbq)$. Let $M$ be a subspace of $\hfr$ which is $\overline{H}:=\Ad(\bbh(\bbq))$-invariant. Then there is a normal subgroup $\bbn$ of $\bbh$ such that $\Lie(\bbn)(\bbq)=M$; in particular, if $M$ is a proper subspace of $\hfr$, then $\bbn$ is a proper normal subgroup of $\bbh$.   
\end{lem}
\begin{proof} 

	Without loss of generality we can and will assume that $\bbh$ is simply-connected. So there are Zariski-connected, simply-connected, semisimple, $\bbq$-simple groups $\bbh_i$ such that $\bbh\simeq \oplus_i \bbh_i$. Without loss of generality we can and will identify $\bbh$ with $\oplus_i \bbh_i$. Therefore $\hfr_i:=\Lie(\bbh_i)(\bbq)$ are simple $\overline{H}$-modules, and $M$ can be identified with a subspace of $\oplus_i \hfr_i$.
	
{\bf Claim.} Let ${\rm pr}_j:\oplus_i \hfr_i\rightarrow \hfr_j$ be the projection onto the $j$-th component. Let $ J:=\{j|\h {\rm pr}_j(M)\neq 0\}$. Then $M=\oplus_{j\in J} \hfr_j$.

{\em Proof of Claim.} For $j\in J$, there is $x=(x_i)_i\in \oplus_i \hfr_i$ such that $x_j\neq 0$. So there is $h_j\in \bbh_j(\bbq)$ such that $\Ad(h_j)(x_j)\neq x_j$. Hence 
\[
\Ad(h_j)(x)-x=\Ad(h_j)(x_j)-x_j\in (M\cap \hfr_j)\setminus \{0\}.
\]
So $M\cap \hfr_j\neq 0$. Since $\hfr_j$ is a simple $\overline{H}$-module, we get $\hfr_j\subseteq M$. Thus $M=\oplus_{j\in J} \hfr_j$.

Now it is clear that the subgroup $\bbn:=\oplus_{j\in J}\bbh_j$ satisfies all the mentioned conditions.
\end{proof}

The next lemma is tightly related to \cite[Lemma 13]{SGV}, where normal subgroups of a perfect algebraic group is described. And we closely follow the proof of the mentioned result. 

\begin{lem}\label{l:AbelianUnipotentRadical}
	Let $\bbh$ be a Zariski-connected semisimple $\bbq$-group. Let $\rho:\bbh\rightarrow \mathbb{GL}(\bbv)$ be a $\bbq$-representation of $\bbh$. Suppose $\bbv(\overline{\bbq})$ has no non-zero $\bbh(\overline{\bbq})$-fixed point. Let $\bbg:=\bbh\ltimes \bbv$, and $\gfr:=\Lie(\bbg)(\bbq)$. Let $M$ be a subspace of $\gfr$ which is $\overline{G}:=\Ad(\bbg(\bbq))$-invariant. Then there is a normal subgroup $\bbn$ of $\bbg$ such that $M=\Lie(\bbn)(\bbq)$.     
\end{lem}
\begin{proof}	
	Let $M':=M\cap V$, and $\underline{M}'$ be the $\bbq$-subgroup of $\bbv$ which is induced by $M'$. Since $\bbh$ is semisimple, $\bbh(\bbq)$ is Zariski-dense in $\bbh$. Hence $\underline{M}'$ is invariant under the action of $\bbh$. 
	
	Passing to $\bbh\ltimes \bbv/\underline{M}'$ and $M/M'\subseteq \hfr\oplus V/M'$, we can and will assume that $V\cap M=\{0\}$. Thus projection to $\hfr$ induces an embedding and we get an $\Ad(\bbh(\bbq))$-module homomorphism $\phi:\pr M\rightarrow V$, where $\pr: \hfr\oplus V\rightarrow \hfr$ is the projection map, such that 
	\[
	M=\{(x,\phi(x))|\h x\in \pr M\}.
	\] 
	For any $v\in V$ and $x\in \pr M$, we have 
	\be\label{e:AdUnipAction}
	\Ad(1,v)(x,\phi(x))-(x,\phi(x))=(0,[x,v])\in M.
	\ee
	
	Hence by (\ref{e:AdUnipAction}) we have $[\pr M, V]=0$. On the other hand, by Lemma~\ref{l:Semisimple}, there is a normal subgroup $\bbh_M$ of $\bbh$ such that $\pr M=\Lie(\bbh_M)(\bbq)$. In particular, $\bbh_M$ is a semisimple $\bbq$-group. So we have that $\bbh_M$ acts trivially on $\bbv$. Thus $\bbh_M$ is a normal subgroup of $\bbg$.  
	
	For any $h$ in the centralizer $C_{\bbh(\bbq)}(\bbh_M(\bbq))$ of $\bbh_M(\bbq)$ in $\bbh(\bbq)$ and $x\in \pr M$, we have $\Ad(h)(\phi(x))=\phi(\Ad(h)(x))=\phi(x)$. On the other hand, since $\bbh$ is a semisimple $\bbq$-group and $\bbh_M$ is a normal $\bbq$-subgroup of $\bbh$, we have $\bbh(\bbq)=\bbh_M(\bbq)C_{\bbh(\bbq)}(\bbh_M(\bbq))$. Overall we get that $\bbh(\bbq)$ acts trivially on $\phi(\pr M)$. Hence $\phi=0$, and we get $M=(\Lie \bbh_M)(\bbq)$. 	
\end{proof}
\begin{proof}[Proof of Proposition~\ref{p:GeneratingLieAlgebra}]
	By Lemma~\ref{l:Perfect}, $\bbg_1$ is perfect. So $\bbg_1\simeq \bbh \ltimes \bbu$ where $\bbh$ is a semisimple $\bbq$-group $\bbh$  and $\bbu$ is a unipotent $\bbq$-group. Moreover $\bbv(\overline{\bbq})$ has no non-zero $\Ad(\bbh(\overline{\bbq}))$-fixed element where $\bbv=\bbu/[\bbu,\bbu]$. 
	
	Let $M$ be the $\Ad(\bbg_1(\bbq))$-module generated by $\gfr_2$. We want to show that $M=\gfr_1$. By Lemma~\ref{l:ReductionToAbelianUnipotentRadical} and Nakayama's lemma, it is enough to prove $[\ufr,\ufr]+M=\gfr_1$ where $\ufr:=\Lie (\bbu)(\bbq)$.
	
	By applying Lemma~\ref{l:AbelianUnipotentRadical} for $\bbh\ltimes \bbv$ and 
	\[
	(M+[\ufr,\ufr])/[\ufr,\ufr]\subseteq \Lie(\bbh)(\bbq)\oplus \Lie(\bbv)(\bbq)=\hfr\oplus \ufr/[\ufr,\ufr],
	\]
 we get that there is a normal $\bbq$-subgroup $\bbn$ of $\bbg_1$ such that $M+[\ufr,\ufr]=\Lie(\bbn)(\bbq)\supseteq \Lie (\bbg_2)(\bbq)$. Hence $\bbn\supseteq \bbg_2$. Since the normal closure of $\bbg_2$ in $\bbg_1$ is $\bbg_1$, we get that $\bbn=\bbg_1$. And so $M+[\ufr,\ufr]=\gfr_1$, which finishes the proof as it was explained above.  	
\end{proof}

\section{Simultaneous bounded generation of almost all the fibers: Zariski-topology.}\label{s:Zariski}

In this section we prove a {\em bounded generation} statement at the level of Zariski-topology (see Proposition~\ref{p:BGZariskiTopology}). 

To avoid recalling the definition of some well-known terms {\em within} the statements, they are mentioned here. Along the way, some needed notation is introduced.  

Let $\Spec(\bbz[1/q_0])$ be the affine scheme of the ring $\bbz[1/q_0]$; in particular, the set of its points is 
\[
\Spec(\bbz[1/q_0])=\{0\}\cup\{p\bbz[1/q_0]|\h p \text{ is a prime integer}, p\nmid q_0 \}.
\]
 The point $(0)$ is called the {\em generic} point (as it is dense). For any $p\in \Spec(\bbz[1/q_0])$, the residue field over $p\bbz[1/q_0]$ is denoted by $k(p)$; that means $k(p)$ is either the finite field with $p$ elements if $p\neq 0$, or $\bbq$ if $p=0$. For any field $F$, its algebraic closure is denoted by $\overline{F}$. 
 
 Here we have to work with {\em affine, finite type, reduced, flat} group schemes $\gcal$ over $\bbz[1/q_0]$; that means, as a scheme $\gcal=\Spec(A)$ where $A$ is a finitely generated $\bbz[1/q_0]$-algebra with no non-zero nilpotent element and no additively torsion element, and in addition $A$ has a Hopf algebra structure.        For an affine group scheme $\gcal=\Spec(A)$ over $\bbz[1/q_0]$, its fiber over $p\in \Spec(\bbz[1/q_0])$ is denoted by $\gcal^{(p)}$; that means $\gcal^{(p)}=\Spec(A\otimes_{\bbz[1/q_0]}k(p))$. For a group scheme $\gcal$ over $\bbz[1/q_0]$, its geometric fiber over $p\in \Spec(\bbz[1/q_0])$ is denoted by $\bbg^{(p)}$; that means $\bbg^{(p)}:=\Spec(A\otimes_{\bbz[1/q_0]}\overline{k(p)})$. The fiber $\gcal^{(0)}$ over $(0)$ is called the generic fiber, and the geometric fiber $\bbg^{(0)}$ is sometimes denoted by $\bbg$. A morphism $f:\Spec(A)\rightarrow\Spec(B)$ between two schemes is called {\em dominant} if its image is dense. 
 
 For any non-zero integer $q$, any homomorphism induced by the quotient map $\pi_q:\bbz[1/q_0]\rightarrow \bbz[1/q_0]/q\bbz[1/q_0]$ is still denoted by $\pi_q$. We let $\pi_0$ denote all the homomorphisms induced by the embedding $\pi_0:\bbz[1/q_0]\rightarrow \bbq$. In particular, for a group scheme $\gcal$ over $\bbz[1/q_0]$ and $p\bbz[1/q_0]\in \Spec(\bbz[1/q_0])$, we get a group homomorphism $\pi_p:\gcal(\bbz[1/q_0])\rightarrow \gcal^{(p)}(k(p))$. 
 
 The normal closure of an algebraic subgroup $\bbg_2$ of $\bbg_1$ is the smallest normal closed subgroup of $\bbg_2$ in $\bbg_1$. 

\begin{prop}\label{p:BGZariskiTopology}
 Let $\Gamma_2\subseteq \Gamma_1$ be subgroups of $\GL_{n_0}(\bbz[1/q_0])$. 
 For $i=1,2$, let $\gcal_i$ be the Zariski-closure of $\Gamma_i$ in $(\GL_{n_0})_{\bbz[1/q_0]}$, and suppose the geometric generic fiber $\bbg_i$ of $\gcal_i$ is irreducible. Suppose  
 the normal closure of $\bbg_2$ in $\bbg_1$ is $\bbg_1$. Then there are $\gamma_1,\ldots, \gamma_c$ in $\Gamma_1$ such that the following is a surjective morphism
 \[
 f^{(p)}_{(\gamma_i)}(h_1,\ldots,h_c):\bbg_2^{(p)}\times \cdots \times \bbg_2^{(p)} \rightarrow \bbg_1^{(p)}, \hspace{.5cm} f^{(p)}_{(\gamma_i)}(h_1,\ldots,h_c):=\pi_p(\gamma_1)h_1\pi_p(\gamma_1)^{-1}\cdots  \pi_p(\gamma_c)h_c\pi_p(\gamma_c)^{-1},
 \]     
 where $p\bbz[1/q_0]$ ranges in a non-empty open subset of $\Spec(\bbz[1/q_0])$; that means $p\bbz[1/q_0]$ can be any prime ideal except for finitely many non-zero ones. Furthermore $\bbg_i^{(p)}$ are Zariski-connected. 
\end{prop}
Let us recall parts of Proposition 9.6.1, Theorem 9.7.7 and Theorem 12.2.4 from \cite{EGA} where {\em constructibility} of being a dominant morphism (see~\cite[Chapter I.3.3]{DG} for definition of a constructible set) and {\em genericness} of dimension, being smooth, and being geometrically irreducible for the fibers of a $\bbz[1/q_0]$-scheme are proved (see~\cite[Chapter I.4.4]{DG} for definition of a smooth morphism and see~\cite[Theorem 40, Lemma 42]{SGV} for an effective version of the former results). 

\begin{thm}\label{t:GenericnessConstructiblityDominant}
\begin{enumerate}
	\item Let $\vcal$ be an affine $\bbz[1/q_0]$-scheme of finite type. Suppose the generic fiber $\vcal^{(0)}$ of $\vcal$ is smooth and geometrically irreducible; that is to say $\bbv^{(0)}:=\vcal^{(0)}\times_{\Spec(\bbq)}\Spec(\overline{\bbq})$ is irreducible. Then there is $q_1\in \bbz\setminus\{0\}$ such that 
	\begin{enumerate}
		\item For any $p\nmid q_0q_1$, the fiber $\vcal^{(p)}$ of $\vcal$ over $p\bbz[1/q_0]$ is a smooth $k(p)$-scheme,
		\item For any $p\nmid q_0q_1$, $\vcal^{(p)}$ is geometrically irreducible; that is to say $\bbv^{(p)}:=\vcal^{(p)}\times_{\Spec(k(p))}\Spec(\overline{k(p)})$ is irreducible,
		\item For any $p\nmid q_0q_1$, $\dim \vcal^{(0)}=\dim \vcal^{(p)}=\dim \bbv^{(p)}$. 
	\end{enumerate}
	\item Let $\vcal$ and $\wcal$ be two affine, of finite type, $\bbz[1/q_0]$-schemes. Let $f:\vcal\rightarrow \wcal$ be a $\bbz[1/q_0]$-morphism. Let $\bbv^{(p)}$ and $\bbw^{(p)}$ be the geometric fibers over $p\bbz[1/q_0]\in \Spec(\bbz[1/q_0])$ of $\vcal$ and $\wcal$, respectively. Suppose $\bbv^{(0)}$ and $\bbw^{(0)}$ are irreducible and $f^{(0)}:\bbv^{(0)}\rightarrow \bbw^{(0)}$ is dominant. Then there is $q_1\in \bbz\setminus\{0\}$ such that 
	\[
	f^{(p)}:\bbv^{(p)}\rightarrow \bbw^{(p)}
	\]  
	is dominant for any $p\nmid q_0q_1$.
\end{enumerate}
\end{thm}
\begin{proof}
	As it is mentioned earlier, these are special cases of~\cite[Proposition 9.6.1, Theorem 9.7.7, Theorem 12.2.4]{EGA}. 
\end{proof}
\begin{proof}[Proof of Proposition~\ref{p:BGZariskiTopology}] We start by getting a dominant map on the geometric fiber of the generic point.	
	{\bf Claim 1.} For a given $(\gamma_i)_{i=1}^{c_1}\subseteq \Gamma_1$, if $f^{(0)}_{(\gamma_i)_{i=1}^{c_1}}$ is not dominant, then there are $\gamma_{c_1+1},\ldots,\gamma_{c_2}\in \Gamma_1$ such that $\dim\left(\overline{\im (f^{(0)}_{(\gamma_i)_{i=1}^{c_2}})}\right)>\dim\left(\overline{\im (f^{(0)}_{(\gamma_i)_{i=1}^{c_1}})}\right)$. 
	
	{\em Proof of Claim 1.} Since $\bbg_2$ is irreducible, $\bbv:=\overline{\im (f^{(0)}_{(\gamma_i)})}$ is irreducible. Therefore $\overline{\bbv\cdot \bbv^{-1}\cdot \bbv}$ is irreducible, too. If $\overline{\bbv\cdot \bbv^{-1}\cdot \bbv}\neq \bbv$, then $\dim \overline{\bbv\cdot \bbv^{-1}\cdot \bbv}>\dim \bbv$ and we get the desired claim as 
		\[
		\overline{\bbv\cdot \bbv^{-1}\cdot \bbv}=\overline{\im( f^{(0)}_{(\gamma_i)_{i=1}^{c_2}})},
		\]
		 where $c_2:=3c_1$ and $\gamma_{kc_1+i}:=\gamma_i$ for $0\le k\le 2$ and $1\le i\le c_1$.
	
	So we can and will assume that  $\overline{\bbv\cdot \bbv^{-1}\cdot \bbv}= \bbv$, which implies that $\bbv$ is a subgroup of $\bbg_1$. Since the normal closure of $\bbg_2$ in $\bbg_1$ is $\bbg_1$, $\bbv$ cannot be a normal subgroup of $\bbg_1$. Since $\Gamma_1$ is Zariski-dense in $\bbg_1$, there is $\gamma\in\Gamma_1$ such that $\gamma \bbv \gamma^{-1} \cdot \bbv\neq \bbv$. Again by irreducibility of $\bbv$ and $\overline{\gamma \bbv \gamma^{-1} \cdot \bbv}$ we get the desired claim.
	 
	{\bf Claim 2.} There are $\gamma_1,\ldots,\gamma_{c_1}\in \Gamma_1$ such that $f^{(0)}_{(\gamma_i)}$ is a dominant map.
	
	{\em Proof of Claim 2.} Applying Claim 1 for at most $\dim \bbg_1$ many times, we get $\gamma_1,\ldots,\gamma_{c_{\dim \bbg_1}}\in\Gamma_1$ such that $f^{(0)}_{(\gamma_i)}$ is dominant.
	
	{\bf Claim 3.} There are $q_1\in \bbz\setminus\{0\}$ and $\gamma_1,\ldots,\gamma_{c_1}\in \Gamma_1$ such that for any $p\nmid q_0q_1$ (including $p=0$) we have that $f^{(p)}_{(\gamma_i)}$ is dominant and $\bbg_i^{(p)}$ is irreducible.
	
	{\em Proof of Claim 3.} Let $\gamma_i$'s be as in Claim 2. Now we can get the rest using Theorem~\ref{t:GenericnessConstructiblityDominant} parts 1(b) and (2).
	
	{\bf Claim 4.} There are $\gamma_1,\ldots,\gamma_{c}\in \Gamma_1$ and $q_1\in \bbz\setminus\{0\}$ such that for any $p\nmid q_0q_1$ (including $p=0$) we have that
	$f^{(p)}_{(\gamma_i)}$ is surjective.
	
	{\em Proof of Claim 4.} Since $\overline{k(p)}$ is algebraically closed, it is enough to find $\gamma_i$ such that $f^{(p)}_{(\gamma_i)}$ induces a surjection from $\bbg_2^{(p)}(\overline{k(p)})\times \cdots \times \bbg_2^{(p)}(\overline{k(p)})$ to $\bbg_1^{(p)}(\overline{k(p)})$ (see~\cite[Corollary 11, Chapter I.3.6]{DG}). Let $\gamma_1,\ldots,\gamma_{c_1}$ be as in Claim 3. So by Chevalley's theorem~\cite[Chapter 4.4]{Hum}, there is an open and dense subset $U_p\subseteq \bbg_1^{(p)}(\overline{k(p)})$ of $\im(f^{(p)}_{(\gamma_i)})$. Hence by \cite[Chapter 7.4]{Hum} we have $\bbg_1^{(p)}(\overline{k(p)})=U_p\cdot U_p$. Therefore $f^{(p)}_{(\gamma_i)_{i=1}^{c}}$ is surjective where $c:=2c_1$ and $\gamma_{c_1+i}:=\gamma_i$ for any $1\le i \le c_1$.
\end{proof}

\section{Large congruence subgroup: for an arbitrary place.}\label{s:LargeConrguence}

In this section, we will prove a bounded generation statement in $p$-adic setting for {\em arbitrary $p$} (see Proposition~\ref{p:LargeCongruence}). 

Besides some well-known terms from algebraic geometry that has been already introduced in Section~\ref{s:Zariski}, we need to recall some terms from algebraic group theory. An algebraic group $\bbg$ is called {\em perfect} if there is no non-trivial homomorphism to an abelian algebraic group; that is to say the derived subgroup $[\bbg,\bbg]$ is equal to $\bbg$. A perfect group $\bbg$ is isomorphic to $\bbl\ltimes \bbu$ where $\bbl$ is a semisimple group and $\bbu$ is a unipotent group. We say a perfect group $\bbg$ is simply connected if its semisimple part is simply connected (see~\cite[Section 3]{SGV} for relevant basic properties of perfect algebraic groups). 

To avoid repeating a series of assumptions for the statements of this section, we record them here for the future reference. 

\begin{itemize}
	\item[(A1)] $\Gamma_2\subseteq \Gamma_1$ are finitely generated subgroups of $\GL_{n_0}(\bbz[1/q_0])$.
	\item[(A2)] $\gcal_i$ is the Zariski-closure of $\Gamma_i$ in $(\GL_{n_0})_{\bbz[1/q_0]}$. 
	\item[(A3)] For any $p\in \Spec(\bbz[1/q_0])$, $\gcal_i^{(p)}$ is the fiber over $p$ of $\gcal_i$.
	\item[(A4)] For any $p\in \Spec(\bbz[1/q_0])$, $\bbg_i^{(p)}$ is the geometric fiber of $\gcal_i^{(p)}$.
	\item[(A5)] $\bbg_i:=\bbg_i^{(0)}$ are perfect and $\bbg_1$ is simply-connected.
	\item[(A6)] For $\gamma_i\in\Gamma_1$ and $p\in \Spec(\bbz[1/q_0])$, $f_{p,(\gamma_i)}:\gcal_2^{(p)}\times \cdots \times \gcal_2^{(p)}\rightarrow \gcal_1^{(p)}$ and $f^{(p)}_{(\gamma_i)}: \bbg_2^{(p)}\times \cdots \times \bbg_2^{(p)}\rightarrow \bbg_1^{(p)}$ be the morphisms that are given by 
	\[
	(h_i)\mapsto \pi_p(\gamma_1)h_1\pi_p(\gamma_1)^{-1}\cdots\pi_p(\gamma_c)h_1\pi_p(\gamma_c)^{-1}.
	\]
	\item[(A7)] For some $\gamma_1,\ldots,\gamma_{m_0}\in \Gamma_1$, $f^{(p)}_{(\gamma_i)}$ is surjective if $p$ is large enough. 
\end{itemize}

\begin{prop}\label{p:LargeCongruence}
	In the setting of (A1)-(A7), there are a positive integer $N$ and $\gamma_{m_0+1},\ldots,\gamma_{m_1}\in \Gamma$,such that for any prime $p\nmid q_0$ we have 
	\[
	\Gamma_{1,p}[p^N]\subseteq \gamma_1 \Gamma_{2,p} \gamma_1^{-1} \cdots \gamma_{m_1} \Gamma_{2,p} \gamma_{m_1}^{-1},
	\] 
	where $\Gamma_{i,p}$ is the closure of $\Gamma_i$ in $\GL_{n_0}(\bbz_p)$ and $\Gamma_{1,p}[p^N]:=\Gamma_{1,p}\cap \ker(\gcal_1(\bbz_p)\xrightarrow{\pi_{p^N}}\gcal_1(\bbz/p^N\bbz))$.
\end{prop}
\begin{lem}\label{l:LieAlgebraExpansion}
	In the setting of (A1)-(A7), there are $\dim \bbg_1$-many elements $\gamma_i'$ in $\Gamma_1$ such that 
	\[
	\Lie(\bbg_1)(\bbq)=\sum_{i=1}^{\dim \bbg_1} \Ad(\gamma_i')(\Lie (\bbg_2)(\bbq)).
	\]
And so $\sum_{i=1}^{\dim \bbg_1} \Ad(\gamma_i')(\Lie (\gcal_2)(\prod_{q_0\nmid p}\bbz_p))$ is an open subgroup of  $\Lie(\gcal_1)(\prod_{q_0\nmid p}\bbz_p)$.
\end{lem}
\begin{proof}
	The first part is a corollary of Proposition~\ref{p:GeneratingLieAlgebra} and the assumption that $\Gamma_1$ is Zariski-dense in $\bbg_1$. The second part is consequence of Chinese remainder theorem and the first part.
\end{proof}

\begin{lem}\label{l:StrongApproxLocalChart}
 	In the setting of (A1), (A2), and (A5), for any $p\nmid q_0$, let $\Gamma_{1,p}$ be the closure of $\Gamma_1$ in $\GL_{n_0}(\bbz_p)$, and $\wh{\Gamma}_1$ be the closure of $\Gamma_1$ in $\prod_{p\nmid q_0}\GL_{n_0}(\bbz_p)$. Then the following holds:
 	\begin{enumerate}
 		\item (Strong Approximation: simply-connected) $\wh{\Gamma}_1$ is an open subgroup of $\prod_{p\nmid q_0}\gcal_1(\bbz_p)$; in particular, there is a positive integer $c_0$ such that for any prime $p\nmid q_0$ and any integer $c\ge c_0$, we have $\Gamma_{1,p}[p^{c}]=\gcal_1(\bbz_p)[p^{c}]$ where as before these are principal congruence subgroups. 
 		\item (Strong Approximation: general) There is a positive integer $c_1$ such that for any integer $c\ge c_1$ we have
 		$\gcal_2(\bbz_p)[p^c]=\Gamma_{2,p}[p^{c}]$. Moreover, if $p$ is large enough, then $\gcal_2(\bbz_p)[p]=\Gamma_{2,p}[p]$
 		\item (Local charts) There is a positive integer $c_2$ such that for any integer $c\ge c_2$, the exponential and the logarithmic maps induce bijections between $\gcal_i(\bbz_p)[p^{c}]$ and $\Lie(\gcal_i)(\bbz_p)\cap p^c\gl_{n_0}(\bbz_p)$ for $i=1$ or $2$. 
 	\end{enumerate}
 \end{lem}
\begin{proof}
	The first part is Nori's strong approximation~\cite[Theorem 5.4]{Nor}. The third part is a direct consequence of Proposition~\ref{p:AnalyticIsAlgebraic} and Lemma~\ref{l:FlatModel}.
	
	Let $\wt{\bbg}_2$ be the simply-connected cover of $\bbg_2$, and $\iota:\wt{\bbg}_2\rightarrow \bbg_2$ be the $\bbq$-central isogeny. Let $\Lambda:=\iota^{-1}(\Gamma_2)\cap \wt{\bbg}_2(\bbq)$. Then, as in \cite[Lemma 24]{SGS}, we have that $\Lambda$ is a finitely generated, Zariski-dense subgroup of $\wt{\bbg}_2$. Hence, by Nori's strong approximation, we have that the closure $\Lambda_p$ of $\Lambda$ in $\iota^{-1}(\gcal_2(\bbz_p))$ is open; and moreover, for large enough $p$, $\Lambda_p=\iota^{-1}(\gcal_2(\bbz_p))$. Since $|\bbg_2(\bbq_p)/\iota(\wt{\bbg}_2(\bbq_p))|\le p^{c_3}$ where $c_3$ just depends on the dimension of $\bbg_2$, we have that $\iota(\Lambda_p)\supseteq \gcal_2(\bbz_p)[p^{c_1}]$ where $c_1$ just depends on $\Gamma_2\subseteq \GL_{n_0}(\bbz[1/q_0])$. Therefore $\Gamma_{2,p}\supseteq \gcal_2(\bbz_p)[p^{c_1}].$ 
	
	To show the last claim of part (2), we notice that $\mu:=\ker(\iota)$ is a central subgroup of $\wt{\bbg}_2$ which is a $\bbq_p$-group. Using Galois cohomology, we get the following exact sequence:
	\[
	1\rightarrow \mu(\bbq_p)\rightarrow \wt{\bbg}_2(\bbq_p)\rightarrow \bbg(\bbq_p) \rightarrow H^1(\bbq_p,\mu).
	\] 
	Hence $\bbg(\bbq_p)/\iota(\wt{\bbg}_2(\bbq_p))$ can be embedded into $H^1(\bbq_p,\mu)$, which is an abelian $m$-torsion group for some $m$ that can be bounded by the dimension of $\bbg_2$. In particular, for large enough $p$, $\bbg(\bbq_p)/\iota(\wt{\bbg}_2(\bbq_p))$ does not have any $p$-element. As $\gcal_2(\bbz_p)[p]/\iota(\iota^{-1}(\gcal_2(\bbz_p)))[p]$ can be embedded into $\bbg(\bbq_p)/\iota(\wt{\bbg}_2(\bbq_p))$, we get that $\gcal_2(\bbz_p)[p]/\iota(\iota^{-1}(\gcal_2(\bbz_p)))[p]$ has no $p$-element. On the other hand, $\gcal_2(\bbz_p)[p]$ is a pro-$p$ group, and so all of its finite quotients are $p$-groups. Therefore we get 
	\be\label{e:LargeIotaImage}
	\gcal_2(\bbz_p)[p]=\iota(\iota^{-1}(\gcal_2(\bbz_p)))[p].
	\ee
	As we said earlier, $\Lambda_p=\iota^{-1}(\bbg_2(\bbz_p))$, for large enough $p$. Thus by (\ref{e:LargeIotaImage}) we have
	\[
	\gcal_2(\bbz_p)[p]=\Gamma_{2,p}[p].
	\]
\end{proof}
\begin{proof}[Proof of Proposition~\ref{p:LargeCongruence}] Let $m_1:=m_0+\dim \bbg_1$ and $\gamma_{m_0+i}:=\gamma_i'$ where $\gamma_i'$ are the elements given by Lemma~\ref{l:LieAlgebraExpansion}. Let $c$ be an integer larger than $\max\{c_0,c_1,c_2\}$ where $c_i$'s are given in Lemma~\ref{l:StrongApproxLocalChart}. Let $\gfr_{i,p}:=\Lie(\gcal_i)(\bbz_p)$. Let $F:p^c \gfr_{2,p} \times \cdots \times p^c \gfr_{2,p} \rightarrow p^c \gfr_{1,p}$ be the composite of the following $p$-adic analytic functions: 
\[
\begin{array}{lll}
p^c \gfr_{2,p} \times \cdots \times p^c \gfr_{2,p}& 
\xrightarrow{\exp} \gcal_2(\bbz_p)[p^c]\times \cdots \times \gcal_2(\bbz_p)[p^c],
& \exp(x_1,\ldots, x_{m_1}):=(\exp x_1,\ldots, \exp x_{m_1}),\\
& \xrightarrow{{\rm Conj}}\gcal_1(\bbz_p)[p^c]\times \cdots \times \gcal_1(\bbz_p)[p^c], 
& {\rm Conj}(h_1,\ldots,h_{m_1}):=(\gamma_1h_1\gamma_1^{-1},\ldots, \gamma_{m_1}h_{m_1}\gamma_{m_1}^{-1}),\\
& \xrightarrow{{\rm Prod}} \gcal_1(\bbz_p)[p^c], 
&{\rm Prod}(g_1,\ldots,g_{m_1}):=g_1\cdot \cdots \cdot g_{m_1},\\
& \xrightarrow{\log} \gfr_{1,p}\cap p^c\gl_{n_0}(\bbz_p), 
& g\mapsto \log(g).
\end{array}
\]
So by the chain rule we have $dF({\bf 0}):\gfr_{2,p}\times \cdots \times \gfr_{2,p}\rightarrow \gfr_{1,p}$ is the composite of the following maps:
\[
\begin{array}{lll}
\gfr_{2,p} \times \cdots \times \gfr_{2,p}& 
\xrightarrow{d(\exp)({\bf 0})} \gfr_{2,p}\times \cdots \times \gfr_{2,p},
& d(\exp)({\bf 0})(x_1,\ldots, x_{m_1}):=(x_1,\ldots, x_{m_1}),\\
& \xrightarrow{d({\rm Conj})(I)}\gfr_{1,p}\times \cdots \times \gfr_{1,p}, 
& d({\rm Conj})(I)(x_1,\ldots,x_{m_1})=(\Ad(\gamma_1)(x_1),\ldots, \Ad(\gamma_{m_1})(x_{m_1})),\\
& \xrightarrow{d({\rm Prod})(I)} \gfr_1, 
&d({\rm Prod})(I)(y_1,\ldots,y_{m_1})=y_1+\cdots+y_{m_1},\\
& \xrightarrow{d(\log)(I)} \gfr_{1,p}, 
& d(\log)(I)(y)=y.
\end{array}
\] 
Hence by Lemma~\ref{l:LieAlgebraExpansion} and Lemma~\ref{l:StrongApproxLocalChart}, we have that, for large enough $p$, $dF({\bf 0})$ is onto, and for any $p\nmid q_0$ the image of $dF({\bf 0})$ is open in $\gfr_{1,p}$. Hence for some non-zero integer $a_0$ we have $N(dF({\bf 0}))\ge |a_0|_p$ where, for a $d$-by-$m$ matrix $X=[\vbf_1 \cdots \vbf_m]$, $N(X):=\max_{1\le i_1\le \cdots\le i_d\le m}|\det[\vbf_{i_1}\cdots \vbf_{i_d}]|$. 

By defintion of $F$, we have $F(x_1,\ldots, x_{m_1})=\log(\exp(\Ad(\gamma_1)(x_1))\cdots\exp(\Ad(\gamma_{m_1})(x_{m_1}))$. Hence by Baker-Campbell-Hausdorff formula~\cite[Chapter V.5]{Jac} we have 
\[
F(x_1,\ldots,x_{m_1})=\sum_i \Ad(\gamma_i)(x_i)+\sum_{\ibf} c_{\ibf} L_{\ibf}(\Ad(\gamma_1)(x_1),\ldots,\Ad(\gamma_{m_1})(x_{m_1})),
\]
where $L_{\ibf}(y_1,\ldots,y_{m_1})$ is a Lie monomial with multi-index $\ibf$ and $c_{\ibf}\in \bbq$; moreover using the explicit Baker-Campbell-Hausdorff formula we see that 
\be \label{e:NormOfCoefficients}
|c_{\ibf}|_p\le p^{m_2\|\ibf\|_1},
\ee
where $\|(i_1,\ldots,i_{m_1})\|_1=\sum_j i_j$ and $m_2$ is a constant which depends on $m_1$ (independent of $p$). Let $F_1:\gfr_2\times \cdots \times \gfr_2 \rightarrow \gfr_1$ be $F_1(\xbf):=F(p^{cm_2}\xbf)$. Choosing $\bbz_p$-basis for $\gfr_1$ and $\gfr_2$, we identify them by $\bbz_p^{d_1}$ and $\bbz_p^{d_2}$, where $d_i=\dim \bbg_i$. Now using the mentioned Baker-Campbell-Hausdorff formula, writing the Taylor expansion of $F'$ at ${\bf 0}$ with respect to the chosen coordinates we get $F_1(\xbf)=\sum_{\ibf}(c_{\ibf,1}\xbf^{\ibf},\ldots,c_{\ibf,d_1}\xbf^{\ibf})$ and $|c_{\ibf,j}|_p\le 1$ for any $\ibf$ and $j$. We also have $N(dF_1({\bf 0}))\ge |a_0 p^{m_2d_1}|$. Therefore by \cite[Lemma 45']{SG:SAII} there is $l_0$ which is independent of $p$ such that 
\[
F_1({\bf 0})+p^{l}\bbz_p^{d_1}\subseteq F_1(\bbz_p^{d_2m_1}),
\]
which implies that 
\be\label{e:LargeBall} 
p^l\gfr_{1,p}\subseteq \log(\gamma_1\gcal_2(\bbz_p)[p^c]\gamma_1^{-1}\cdots \gamma_{m_1}\gcal_2(\bbz_p)[p^c]\gamma_{m_1}^{-1}).
\ee
By Lemma~\ref{l:StrongApproxLocalChart} and Equation (\ref{e:LargeBall}) we have 
\[
\Gamma_{1,p}[p^l]\subseteq \gamma_1\gcal_2(\bbz_p)[p^c]\gamma_1^{-1}\cdots \gamma_{m_1}\gcal_2(\bbz_p)[p^c]\gamma_{m_1}^{-1}
\subseteq \gamma_1\Gamma_{2,p}\gamma_1^{-1}\cdots \gamma_{m_1}\Gamma_{2,p}\gamma_{m_1}^{-1}.
\]
\end{proof}

\section{Bounded generation: Large primes.}
In this section, we will prove a bounded generation statement in $p$-adic setting for large $p$ (see Proposition~\ref{p:BoundedGenerationLargePrime}).

We will be using all the mentioned terms from algebraic geometry and algebraic group theory; in particular, we will operate under the assumptions (A1)-(A7) which are mentioned in Section~\ref{s:LargeConrguence}. In addition, we assume $\gamma_i$'s satisfy Proposition~\ref{p:LargeCongruence} and Lemma~\ref{l:LieAlgebraExpansion}, which means
\begin{itemize}
	\item[(A8)] We have $\Lie(\bbg_1)(\bbq)=\sum_{i=1}^{c} \Ad(\gamma_i)(\Lie (\bbg_2)(\bbq)).$
	\item[(A9)] For a positive integer $N$ and any prime $p\nmid q_0$, we have 
	$
	\Gamma_{1,p}[p^N]\subseteq \gamma_1 \Gamma_{2,p} \gamma_1^{-1} \cdots \gamma_{c} \Gamma_{2,p} \gamma_{c}^{-1}.
	$
\end{itemize}
It is worth mentioning that, as we have seen in the proof of Proposition~\ref{p:LargeCongruence}, (A1)-(A8) implies (A9). 

\begin{prop}\label{p:BoundedGenerationLargePrime}
In the setting of (A1)-(A9), there is a positive integer $C$ such that for any large enough $p$ we have 
\be\label{e:PAdicBG}
\textstyle
\prod_C (\gamma_1 \Gamma_{2,p} \gamma_1^{-1}\cdots \gamma_c \Gamma_{2,p}\gamma_c^{-1})=\Gamma_{1,p},
\ee
where $\Gamma_{i,p}$ is the closure of $\Gamma_i$ in $\GL_{n_0}(\bbz_p)$ and $\prod_C X=\{x_1\cdots x_C|\h x_i\in X\}$.	 
\end{prop} 
It is enough to prove that the equality holds modulo all the powers of $p$. We start with proving that (\ref{e:PAdicBG}) holds modulo $p$ for $C=3\dim \bbg_1$.  
\begin{lem}\label{l:BGModP}
 In the setting of (A1)-(A8), for large enough $p$, we have
 \[
 \textstyle
 \prod_{3\dim \bbg_1}\pi_p(\gamma_1 \Gamma_{2,p} \gamma_1^{-1}\cdots \gamma_c \Gamma_{2,p}\gamma_c^{-1})=\pi_p(\Gamma_{1,p}),
 \]	
 where $\Gamma_{i,p}$ is the closure of $\Gamma_i$ in $\GL_{n_0}(\bbz_p)$ and $\prod_C X=\{x_1\cdots x_C|\h x_i\in X\}$.
\end{lem}
\begin{proof}
	By (A7), for large enough $p$, $f_{p,(\gamma_i)_{i=1}^c}$ is dominant. By (A5) and Proposition~\ref{p:BGZariskiTopology}, for large enough $p$, $\bbg_2^{(p)}$ is irreducible. So by~\cite[Proposition 2.4]{PR} fibers of $f_{p,(\gamma_i)_{i=1}^{cd_1}}$ have dimension at most $d_1(cd_2-1)$, where $d_1=\dim \bbg_1$, $d_2=\dim \bbg_2$,  $\gamma_{rc+i}=\gamma_i$ for $0\le r<d_1$ and $1\le i\le c$. 
	
	Clearly $f_p:=f_{p,(\gamma_i)_{i=1}^{cd_1}}$ is dominant. Hence, by \cite[Proposition 2.5]{PR}, there is a positive constant $C_1$ such that for any prime $p$ and any $y\in \gcal_1^{(p)}(k(p))$ we have
	\be\label{e:UpperBoundOnFibers}
	|f_p^{-1}(y)(k(p))|\le C_1 p^{\dim(f_p^{-1}(y))}.
	\ee 
	Since any fiber of $f$ has dimension at most $d_1(cd_2-1)$, by (\ref{e:UpperBoundOnFibers}) we get
	\be\label{e:1LowerBoundForTheImage}
	|\pi_p(\Gamma_{2,p})|^{cd_1}\le \sum_{y\in f_p\big(\underbrace{\pi_p(\Gamma_{2,p})\times \cdots \times \pi_p(\Gamma_{2,p})}_\text{$cd_1$ times}\big)}	|f_p^{-1}(y)(k(p))| \le C_1  p^{d_1(cd_2-1)} |f_p(\Gamma_{2,p})|.
	\ee
	
	Since $\bbg_2$ is perfect, by Nori's strong approximation (see~\cite[Section 3]{SGV}) and the well-known order of finite quasi-simple groups of Lie type we have that there is a positive constant $C_2$ such that $|\pi_p(\Gamma_{2,p})|\ge C_2^{-1} p^{d_2}$ for any large enough prime $p$.  So together with (\ref{e:1LowerBoundForTheImage}) we get
	\be\label{e:2LowerBoundForTheImage}
	\textstyle
	C_2^{-cd_1}C_1^{-1} p^{d_1}\le |f_p\big(\underbrace{\pi_p(\Gamma_{2,p})\times \cdots \times \pi_p(\Gamma_{2,p})}_\text{$cd_1$ times}\big)|=|\prod_{d_1}\pi_p(\gamma_1 \Gamma_{2,p} \gamma_1^{-1}\cdots \gamma_c \Gamma_{2,p}\gamma_c^{-1})|. 
	\ee
	Since $\bbg_1$ is perfect and $\bbg_1$ is simply-connected, again by Nori's strong approximation (see~\cite[Section 3]{SGV}) and the well-known order of finite quasi-simple groups of Lie type we have that there is a constant $C_3$ such that $|\pi_p(\Gamma_{1,p})|\ge C_3^{-1} p^{d_2}$ for any large enough prime $p$. Hence by (\ref{e:2LowerBoundForTheImage}) we get 
	\be\label{e:3LowerBoundForTheImage}
	\textstyle
	C_4^{-1} |\pi_p(\Gamma_{1,p})|\le |\prod_{d_1}\pi_p(\gamma_1 \Gamma_{2,p} \gamma_1^{-1}\cdots \gamma_c \Gamma_{2,p}\gamma_c^{-1})|,
	\ee
	for some positive constant $C_4$. By Nori's Strong Approximation (see \cite[Theorem 5.4]{Nor} or \cite[Theorem A]{SGV}) we have that  $\pi_p(\Gamma_{1,p})=\lcal_p(k(p))\ltimes \ucal_p(k(p))$ for large enough $p$, where $\lcal_p$ is a semisimple, simply connected, Zariski connected $k(p)$-group with dimension at most $d_1$ and $\ucal_p$ is a unipotent Zariski connected $k(p)$-group  with dimension at most $d_1$. Let $\rho$ be an irreducible (complex) representation of $\pi_p(\Gamma_{1,p})$. By \cite[Corollary 4]{SGV}, the restriction of $\rho$ to $\lcal_p(k(p))$ is not trivial. Hence by~\cite{LS} 
	\be\label{e:QuasiRandom}
	\dim \rho\ge |\lcal_p(k(p))|^{1/C_5}\ge |\pi_p(\Gamma_{1,p})|^{1/C_6}
	\ee
	 for some positive constants $C_5$ and $C_6$ (depending only on $d_1$). By (\ref{e:3LowerBoundForTheImage}), (\ref{e:QuasiRandom}), and a theorem of Gowers~\cite{Gow} (see \cite[Corollary 1]{NP}), it follows that 
	 \[
	 \textstyle
	 \prod_{3d_1}\pi_p(\gamma_1 \Gamma_{2,p} \gamma_1^{-1}\cdots \gamma_c \Gamma_{2,p}\gamma_c^{-1})=\pi_p(\Gamma_{1,p})
	 \]
 for large enough $p$.
\end{proof}
Next we generate the $k$-th grade $\Gamma_{1,p}[p^k]/\Gamma_{1,p}[p^{k+1}]$ using $\gamma_i$'s. Let us recall the connection between the $k$-th grade and the Lie algebra of the group scheme $\gcal_i$.  
\begin{lem}\label{l:ThekthGrade}
	In the setting of (A1)-(A5), for large enough prime $p$ and any positive integer $k$, 
	\[
	\Psi_k:\gcal_i(\bbz_p)[p^k]/\gcal_i(\bbz_p)[p^{k+1}]\rightarrow \gfr_{i,p}/p\gfr_{i,p},\hspace{.75cm} \Psi_k((I+p^kx)\gcal_i(\bbz_p)[p^{k+1}]):=\pi_p(x),
	\]
 where $\gfr_{i,p}:=\Lie(\gcal_i)(\bbz_p)$, is an isomorphism; moreover for any $g\in \gcal_i(\bbz_p)$ and $g'\in \gcal_i(\bbz_p)[p^k]$ we have 
 \[
  \Psi_k(gg'g^{-1}\gcal_i(\bbz_p)[p^{k+1}])=\Ad(\pi_p(g))(\Psi_k(g'\gcal_i(\bbz_p)[p^{k+1}])).
 \]
\end{lem}
\begin{proof}
	By the discussion in~\cite[Section 2.9]{SG:SAI}, we get that 
	\[
		\wt{\Psi}_k:\gcal_i(\bbz_p)[p^k]/\gcal_i(\bbz_p)[p^{k+1}]\rightarrow \Lie(\gcal_i)(\bbz_p/p\bbz_p),\h\h\h \wt{\Psi}_k((I+p^kx)\gcal_i(\bbz_p)[p^{k+1}]):=\pi_p(x),
	\]
 is a well-defined injective group homomorphism. 
 
 	For large $p$, $\gcal_i\times_{\bbz[1/q_0]}\bbz_p$ is a smooth $\bbz_p$-group scheme, and so $\Lie(\gcal_i)(\bbz_p/p\bbz_p)$ is naturally isomorphic to $\gfr_i/p\gfr_i$. So we get that $\Psi_k$ is a well-defined injective group homomorphism. 
 	
 	For any $x\in \gfr_{i,p}$ and $p>2$, by Proposition~\ref{p:AnalyticIsAlgebraic}, $\exp(p^kx)\in \gcal_i(\bbz_p)[p^k]$. By the definition of the exponential function we can see that 
 	\[
 	\Psi_k(\exp(p^kx)\gcal_i(\bbz_p)[p^{k+1}])=\pi_p(x),
 	\] 
 which implies that $\Psi_k$ is a group isomorphism. The other part of Lemma is easy to check. 
\end{proof}
\begin{lem}\label{l:kthGradeGeneration}
	In the setting of (A1)-(A8), for large enough $p$ and any positive integer $k$, we have 
	\[
	(\gamma_1\Gamma_{2,p}[p^k]\gamma_1^{-1}\cdots\gamma_c\Gamma_{2,p}[p^k]\gamma_c^{-1})\Gamma_{1,p}[p^{k+1}]=\Gamma_{1,p}[p^k].
	\]
\end{lem}
\begin{proof}
	By (A8), for large enough $p$, we have $\gfr_{1,p}=\sum_{i=1}^c \Ad(\gamma_i)(\gfr_{2,p})$ where $\gfr_{i,p}:=\Lie(\gcal_i)(\bbz_p)$. Hence, by Lemma~\ref{l:ThekthGrade}, we have 
	\be\label{e:KthLevel}
		(\gamma_1\gcal_{2}(\bbz_p)[p^k]\gamma_1^{-1}\cdots\gamma_c\gcal_{2}(\bbz_p)[p^k]\gamma_c^{-1})\gcal_{1}(\bbz_p)[p^{k+1}]=\gcal_{1}(\bbz_p)[p^k].
	\ee
	On the other hand, by the second part of Lemma~\ref{l:StrongApproxLocalChart}, we have that, for large enough $p$, $\Gamma_{i,p}[p^k]=\gcal_{i}(\bbz_p)[p^k]$, which together with (\ref{e:KthLevel}) implies the claim.  
\end{proof}
\begin{proof}[Proof of Proposition~\ref{p:BoundedGenerationLargePrime}]
Using Lemma~\ref{l:ThekthGrade}, for $k=1,\ldots,N-1$, we get that
\be\label{e:TopNLayers}
\textstyle
\prod_{N-1}(\gamma_1\Gamma_{2,p}[p]\gamma_1^{-1}\cdots \gamma_c\Gamma_{2,p}[p]\gamma_c^{-1})\Gamma_{1,p}[p^N]=\Gamma_{1,p}[p].
\ee
Lemma~\ref{l:BGModP} and Equation (\ref{e:TopNLayers}), we get
\be\label{e:Top}
\textstyle
\prod_{3\dim \bbg_1+N-1}(\gamma_1\Gamma_{2,p}\gamma_1^{-1}\cdots \gamma_c\Gamma_{2,p}\gamma_c^{-1})\Gamma_{1,p}[p^N]=\Gamma_{1,p}.
\ee	
Condition (A9) together with Equation (\ref{e:Top}) imply that
\[
\textstyle
\prod_{3\dim \bbg_1+N}(\gamma_1\Gamma_{2,p}\gamma_1^{-1}\cdots \gamma_c\Gamma_{2,p}\gamma_c^{-1})=\Gamma_{1,p}.
\]
\end{proof}

\section{Bounded generation: Adelic version.}\label{s:BGAdelic}
In this section, we will prove an adelic bounded generation statement (see Theorem~\ref{t:BGAdelic}).

We will be working under the assumptions (A1)-(A5) mentioned at the beginning of Section~\ref{s:LargeConrguence}. 

\begin{thm}\label{t:BGAdelic}
	In the setting of (A1)-(A5), let $\wh{\Gamma}_i$ be the closure of $\Gamma_i$ in $\prod_{p\nmid q_0}\GL_{n_0}(\bbz_p)$. Then there are $\gamma_1,\ldots,\gamma_m\in \Gamma_1$ such that 
	\[
	\gamma_1 \wh{\Gamma}_2 \gamma_1^{-1}\cdots \gamma_m \wh{\Gamma}_2\gamma_m^{-1}\subseteq \wh{\Gamma}_1
	\]
	is an open subgroup of $\wh{\Gamma}_1$.
\end{thm}
\begin{proof}
	As in the proof of Lemma~\ref{l:StrongApproxLocalChart} and \cite[Lemma 24]{SGS}, let $\wt{\bbg}_2$ be the simply-connected cover of $\bbg_2$, and $\iota:\wt{\bbg}_2\rightarrow \bbg_2$ be the $\bbq$-central isogeny. Let 
	$\Lambda:=\iota^{-1}(\Gamma_2)\cap \wt{\bbg}_2(\bbq)$ and $\Gamma_2':=\iota(\Lambda)$. Then, as in \cite[Lemma 24]{SGS}, $\Lambda$ is a finitely generated Zariski-dense subgroup of $\wt{\bbg}_2$; and so $\Gamma_2'$ is a finitely generated Zariski-dense subgroup of $\bbg_2$. We also notice that $\Gamma_2'\subseteq \Gamma_2\subseteq \Gamma_1$.
	
	By Nori's strong approximation, the closure $\wh{\Lambda}$ of $\Lambda$ in $\prod_{p\nmid q_0}\iota^{-1}(\gcal_2(\bbz_p))$ is open. So passing to a finite-index subgroup of $\Lambda$, if needed, we can and will assume that $\wh{\Lambda}=\prod_{p\nmid q_0} \Lambda_p$ where $\Lambda_p$ is the closure of $\Lambda$ in $\iota^{-1}(\gcal_2(\bbz_p))$. Therefore the closure of $\Gamma_2'$ in $\prod_{p\nmid q_0}\gcal_2(\bbz_p)$ is $\iota(\wh{\Lambda})=\prod_{p\nmid q_0}\Gamma'_{2,p}$, where $\Gamma_{2,p}'$ is the closure of $\Gamma_{2}'$ in $\gcal_2(\bbz_p)$. We notice that $\Gamma_2'$ is a finitely generated subgroup of $\Gamma_2$, and it is Zariski-dense in $\bbg_2$ since $\bbg_2$ is Zariski-connected.
	
	By Proposition~\ref{p:BGZariskiTopology} applied to $\Gamma_2'\subseteq \Gamma_1$, we get that $\Gamma_2'\subseteq \Gamma_1$ satisfy (A1)-(A7)  for some $\gamma_i\in \Gamma_1$. Hence by Proposition~\ref{p:LargeCongruence} for some $\gamma_i\in  \Gamma_1$ and positive integer $N$ we have
	\be\label{e:LargeCongruence}
		\Gamma_{1,p}[p^N]\subseteq \gamma_1 \Gamma'_{2,p} \gamma_1^{-1} \cdots \gamma_{m_1} \Gamma'_{2,p} \gamma_{m_1}^{-1}.
	\ee 
	By Lemma~\ref{l:LieAlgebraExpansion} applied to $\Gamma_2'\subseteq \Gamma_1$, there are $\gamma_i$ in $\Gamma_1$ such that (A1)-(A9) hold for $\Gamma_2'\subseteq \Gamma_1$. Therefore there are $\gamma_i$ in $\Gamma_1$ such that, for any large prime $p$, we have
	\be\label{e:LargePrimes}
		\gamma_1 \Gamma_{2,p}' \gamma_1^{-1}\cdots \gamma_{m_2} \Gamma'_{2,p}\gamma_{m_2}^{-1}=\Gamma_{1,p}.
	\ee  
	By Equations (\ref{e:LargeCongruence}) and (\ref{e:LargePrimes}), we have there are $\gamma_i$ in $\Gamma_1$ such that
	\be\label{e:OpenSubset}
		\gamma_1 \wh{\Gamma}_2' \gamma_1^{-1}\cdots \gamma_{m_3} \wh{\Gamma}'_2\gamma_{m_3}^{-1}=\prod_{p\nmid q_0}		
		(\gamma_1 {\Gamma}_{2,p}' \gamma_1^{-1}\cdots \gamma_{m_3} \Gamma'_{2,p}\gamma_{m_3}^{-1})
		\supseteq \prod_{p\nmid q_0, p<p_0} \Gamma_{1,p}[p^N]\cdot \prod_{p\ge p_0} \Gamma_{1,p}.
	\ee
Therefore $\gamma_1 \wh{\Gamma}_2 \gamma_1^{-1}\cdots \gamma_{m_3} \wh{\Gamma}_2\gamma_{m_3}^{-1}$ contains an open normal subgroup (of finite index) of $\wh{\Gamma}_1$. So by multiplying this set by itself finitely many times we get an open subgroup of $\wh{\Gamma}_1$. This means there are $\gamma_i$ in $\Gamma_1$ such that $\gamma_1 \wh{\Gamma}_2 \gamma_1^{-1}\cdots \gamma_m \wh{\Gamma}_2\gamma_m^{-1}$ is an open subgroup of $\wh{\Gamma}_1$.
\end{proof}
\section{Proof of Theorem~\ref{t:main}: inducing super-approximation.}\label{s:FinalProof}
Let us recall that $\bbg_i$ is the Zariski-closure of $\Gamma_i$ in $(\GL_{n_0})_{\bbq}$ and $\bbg_i^{\circ}$ is the Zariski-connected component of $\bbg_i$. Let $\wt{\bbg}_1$ be the simply-connected cover of $\bbg_1^{\circ}$ and  $\iota:\wt{\bbg}_1\rightarrow \bbg_1^{\circ}$ be its covering map. Let $\wt{\Lambda}_1:=\iota^{-1}(\Gamma_i)\cap \wt{\bbg}_1(\bbq)$, 
	$\wt{\Lambda}_2:=\iota^{-1}(\Gamma_2)\cap \wt{\Lambda}_1\cap \wt{\bbg}_2(\bbq)$, where $\wt{\bbg}_2$ is the Zariski-connected component of the Zariski-closure of $\iota^{-1}(\Gamma_2)\cap \wt{\Lambda}_1$ in $\wt{\bbg}_1$. Finally let $\Lambda_i:=\iota(\wt{\Lambda}_i)$. 

Now we will check that $\wt{\Lambda}_2\subseteq \wt{\Lambda}_1$ satisfy conditions (A1) and (A5).

\begin{narrow}[left=.4cm]
\h\h\h\h  {\bf (A1)} First we notice that $\bbg_1^{\circ}(\bbq)\cap \Gamma_1$ is a subgroup of finite-index of $\Gamma_1$, and so it is finitely generated. Since $\bbg_1^{\circ}(\bbq)/\iota(\wt{\bbg}_1(\bbq))$ is an abelian bounded torsion group and $\bbg_1^{\circ}(\bbq)\cap \Gamma_1$ is finitely generated, we have that $\Lambda_1$ is a subgroup of finite-index in $\Gamma_1$. Thus $\Lambda_1$ is finitely generated. Since $\ker(\iota)(\bbq)$ is finite, we get that $\wt{\Lambda}_1$ is finitely-generated.	
		
		By a similar argument we get $\iota(\iota^{-1}(\Gamma_2)\cap \wt{\Lambda}_1)$ is a finite-index subgroup of $\Gamma_2\cap \bbg^{\circ}_1(\bbq)$. Therefore $\Lambda_2$ is a subgroup of finite-index in $\Gamma_2$. Hence it is finitely-generated, and so is $\wt{\Lambda}_2$.

	By \cite[Lemma 11]{SG:SAI}, there is a $\bbq$-embedding $\wt{\bbg}_1\subseteq (\GL_{n_1})_{\bbq}$ such that $\wt{\Lambda}_1\subseteq \GL_{n_1}(\bbz[1/q_0])$. So we have $\wt{\Lambda}_2\subseteq\wt{\Lambda}_1$ are two finitely generated subgroups of $\GL_{n_1}(\bbz[1/q_0])$.
\end{narrow}
	
\begin{narrow}[left=.4cm]	
\h\h\h\h {\bf (A5)} Since $\Gamma_1$ is Zariski-dense in $\bbg_1$, $\Gamma_1\cap \bbg_1^{\circ}(\bbq)$ is Zariski-dense in $\bbg_1^{\circ}$. Since $\Lambda_1$ is a finite-index subgroup of $\Gamma_1\cap \bbg_1^{\circ}(\bbq)$, it is Zariski-dense in $\bbg_1^{\circ}$. So the restriction of $\iota$ to the Zariski-closure of $\wt{\Lambda}_1$ is still surjective. And so $\wt{\Lambda}_1$ is Zariski-dense in $\wt{\bbg}_1$.
	
		Since $\Lambda_2$ is a subgroup of finite-index in $\Gamma_2$, its Zariski-closure is a subgroup of finite index of $\bbg_2$. On the other hand, $\wt{\Lambda}_2$ is Zariski-dense in a Zariski-connected group $\wt{\bbg}_2$. Hence the Zariski-closure of $\Lambda_2=\iota(\wt{\Lambda}_2)$ is a Zariski-connected group. Therefore the Zariski-closure of $\Lambda_2$ is $\bbg_2^{\circ}$. 
		
		Since $\Gamma_2\acts \prod_{p\nmid q_0} \GL_{n_0}(\bbz_p)$ has spectral gap, by \cite[Proposition 8]{SG:SAI}, $\bbg_2^{\circ}$ is perfect. By Lemma~\ref{l:Perfect}, we get that $\bbg_1^{\circ}$ is perfect. Therefore $\wt{\bbg}_i$'s are perfect. 	
\end{narrow}			

Hence by Theorem~\ref{t:BGAdelic} we have that there are $\wt{\lambda}_i\in \wt{\Lambda}_1$ such that
$
\wt{\lambda}_1\wh{\wt{\Lambda}}_2\wt{\lambda}_1^{-1}\cdots\wt{\lambda}_m\wh{\wt{\Lambda}}_2\wt{\lambda}_m^{-1}
$
is an open subgroup of $\wh{\wt{\Lambda}}_1$ where $\wh{\wt{\Lambda}}_i$ is the closure of $\wt{\Lambda}_i$ in $\prod_{p\nmid q_0}\GL_{n_1}(\bbz_p)$. Hence, applying $\iota$ and letting $\lambda_i:=\iota(\wt{\lambda}_i)$, we get that 
$
\lambda_1 \wh{\Lambda}_2\lambda_1^{-1}\cdots \lambda_m \wh{\Lambda}_2\lambda_m^{-1}
$
is an open subgroup of $\wh{\Lambda}_1$ where $\wh{\Lambda}_i$ is the closure of $\Lambda_i$ in $\prod_{p\nmid q_0}\GL_{n_0}(\bbz_p)$.

{\bf Claim 1.} Let $\overline{\Omega}_2$ be a finite symmetric generating set of $\Lambda_2$. Let $\Omega:=\bigcup_{i=1}^m \lambda_i \overline{\Omega}_2 \lambda_i^{-1}$ and $\Lambda:=\langle \Omega\rangle$. Then $\Lambda\acts \prod_{p\nmid q_0}\GL_{n_0}(\bbz_p)$ has spectral gap.

\begin{proof}[Proof of Claim 1.] The closure of $\Lambda$ in $\prod_{p\nmid q_0}\GL_{n_0}(\bbz_p)$ is 
\be\label{e:BG}
\wh{\Lambda}:=\lambda_1 \wh{\Lambda}_2\lambda_1^{-1}\cdots \lambda_m \wh{\Lambda}_2\lambda_m^{-1}.
\ee
 By \cite[Remark 15, (5)]{SG:SAI}, 
$\lambda(\pcal_{\Omega};\wh{\Lambda})=
\sup_{\gcd(q,q_0)=1}\lambda(\pcal_{\pi_q(\Omega)};
\pi_q(\Lambda))$.

On the other hand, by \cite[Section 2.3]{SG:SAI} we have that $\Gamma_2\acts \prod_{p\nmid q_0} \GL_{n_0}(\bbz_p)$ has spectral gap if and only if $\Lambda_2\acts \prod_{p\nmid q_0} \GL_{n_0}(\bbz_p)$ has spectral gap. Hence
\be\label{e:InSmallGroup}
1>\lambda(\pcal_{\overline{\Omega}_2};\wh{\Lambda}_2)=\sup_{\gcd(q,q_0)=1}\lambda(\pcal_{\pi_q(\overline{\Omega}_2)};\pi_q(\Lambda_2)).
\ee
Therefore, by Varj\'{u}'s lemma~\cite[Lemma A.4]{BK}, Equations (\ref{e:BG}) and (\ref{e:InSmallGroup}), we get that $\lambda(\pcal_{\Omega};\wh{\Lambda})<1$.
\end{proof}
{\bf Claim 2.} $\Lambda_1\acts \prod_{p\nmid q_0}\GL_{n_0}(\bbz_p)$ has spectral gap.

\begin{proof}[Proof of Claim 2.] Let $\overline{\Lambda}:=\Lambda_1\cap \wh{\Lambda}$ where $\wh{\Lambda}$ is given in the proof of Claim 1. Since $\wh{\Lambda}$ is a subgroup of finite index of $\wh{\Lambda}_1$, $\overline{\Lambda}$ is a subgroup of finite index of $\Lambda_1$. Therefore by \cite[Section 2.3]{SG:SAI} we have that $\Lambda_1\acts \prod_{p\nmid q_0} \GL_{n_0}(\bbz_p)$ has spectral gap if and only if $\overline{\Lambda}\acts \prod_{p\nmid q_0} \GL_{n_0}(\bbz_p)$ has spectral gap. Since $\Lambda\subseteq \overline{\Lambda}$, both of them are dense in $\wh{\Lambda}$, and $\Lambda\acts \wh{\Lambda}$ has spectral gap, we get that $\overline{\Lambda}\acts \wh{\Lambda}$ has spectral gap. Thus $\Lambda_1\acts \prod_{p\nmid q_0} \GL_{n_0}(\bbz_p)$ has spectral gap.
\end{proof}
Since $\Lambda_1$ is a subgroup of finite index of $\Gamma_1$ and $\Lambda_1\acts \prod_{p\nmid q_0}\GL_{n_0}(\bbz_p)$ has spectral gap, another application of \cite[Section 2.3]{SG:SAI} implies that $\Gamma_1\acts \prod_{p\nmid q_0}\GL_{n_0}(\bbz_p)$ has spectral gap.

\end{document}